\documentclass{amsart}
\usepackage{latexsym,amsxtra,amscd,ifthen}
\usepackage{amsfonts}
\usepackage{verbatim}
\usepackage{amsmath}
\usepackage{amsthm}
\usepackage{amssymb}

\numberwithin{equation}{section}

\theoremstyle{plain}
\newtheorem{theorem}{Theorem}[section]
\newtheorem{lemma}[theorem]{Lemma}
\newtheorem{proposition}[theorem]{Proposition}
\newtheorem{corollary}[theorem]{Corollary}

\theoremstyle{definition}
\newtheorem{definition}[theorem]{Definition}
\newtheorem{example}[theorem]{Example}

\newtheorem{remark}[theorem]{Remark}
\newtheorem*{remark*}{Remark}
\newtheorem{question}[theorem]{Question}

\makeatletter              
\let\c@equation\c@theorem  
\makeatother

\DeclareMathOperator{\gldim}{gldim}

\DeclareMathOperator{\Ext}{Ext}

\DeclareMathOperator{\gr}{gr}

\DeclareMathOperator{\GKdim}{GKdim}

\DeclareMathOperator{\LDA}{{\sf CoLieAlg}}
\DeclareMathOperator{\LA}{{\sf LieAlg}}
\DeclareMathOperator{\HA}{{\sf HopfAlg}}

\newcommand\eps{\epsilon}

\DeclareMathOperator{\h}{H}

\begin{document}

\title[Hopf algebras of GK-dimension 4]
{Connected Hopf algebras of \\Gelfand-Kirillov dimension four}

\author{D.-G. Wang, J.J. Zhang and G. Zhuang}

\address{Wang: School of Mathematical Sciences,
Qufu Normal University, Qufu, Shandong 273165, P.R.China}

\email{dgwang@mail.qfnu.edu.cn, dingguo95@126.com}

\address{Zhang: Department of Mathematics, Box 354350,
University of Washington, Seattle, Washington 98195, USA}

\email{zhang@math.washington.edu}

\address{Zhuang: Department of Mathematics, Box 354350,
University of Washington, Seattle, Washington 98195, USA}

\email{gzhuang@math.washington.edu}

\begin{abstract}
We classify connected Hopf algebras of Gelfand-Kirillov dimension
4 over an algebraic closed field of characteristic zero.
\end{abstract}





\maketitle


\setcounter{section}{-1}
\section{Introduction}
\label{zzsec0}

For the introduction and most part of the paper we assume that the
base field $k$ is algebraic closed of characteristic zero. Noncommutative
Hopf algebras of finite Gelfand-Kirillov dimension (GK-dimension, for
short) have been studied in several papers, see for example,
\cite{AA, AS1, AS2, BZ, GZ2, Li, WZZ1, WZZ2, Zh1, Zh2}. The
third-named author proved that, if a Hopf algebra $H$ is a connected,
then the associated graded Hopf algebra $\gr H$ with respect to the
coradical filtration is isomorphic to a commutative Hopf algebra
\cite[Proposition 6.4]{Zh2}. If $H$ has finite GK-dimension, then $\gr H$ is
isomorphic to the polynomial ring $k[x_1,\cdots,x_n]$ \cite[Theorem
6.10]{Zh2}, namely, the regular functions ${\mathcal O}(G)$ on a
unipotent group $G$, or equivalently, the graded dual $U(\mathfrak
L)^*$ of the universal enveloping algebra over a graded Lie algebra
$\mathfrak L$, which is called the {\it lantern} of $H$.  Since the
coradical filtration is naturally associated to the given Hopf
algebra, the lantern $\mathfrak L$, as well as the associated
unipotent group $G$, are invariants of $H$.
This observation motivates the following two
related questions.

\begin{question}
\label{zzque0.1} What are the invariants of $H$ that determine
completely the Hopf algebra structure of $H$?
\end{question}

\begin{question}
\label{zzque0.2} Can we classify all connected Hopf algebras of
finite GK-dimension?
\end{question}

The first question was suggested by Andruskiewitsch and Brown. Also
see the talk given by Brown at the Banff workshop \cite{Br3}. The
second question is a subquestion of several motivating questions for
a couple of ongoing classification projects initiated by many people
such as Andruskiewitsch, Schneider, Brown, Goodearl and their
collaborators. Some studies of general Hopf algebras of GK-dimension
1 and 2 are given in \cite{BZ, Li, GZ2, WZZ2} by using homological tools.
There is little chance to list all isomorphism classes of connected
Hopf algebras. A more practical question is along the line of Question
\ref{zzque0.1}: can we classify connected Hopf algebras in terms of
their invariants?

The connected hypothesis is quite restrictive, but there are some
interesting new Hopf algebras in this class even when the GK-dimension
is 3, see the classification of connected Hopf algebras of GK-dimension
3 in \cite[Theorem 1.3]{Zh2}. Since $H$ is a deformation of $\gr H$, it
is possible to understand all $H$ in Question \ref{zzque0.2} if the
(Hopf)-cohomologies of $\gr H$ can be worked out completely. But we will
not consider this in the present paper.

The first goal of the paper is to provide some invariants of
connected Hopf algebras that help us to understand partially the
structure of the Hopf algebra. One of such is the coassociative Lie
algebra that was introduced in \cite{WZZ3}. Our second goal is to
study and classify all connected Hopf algebras of GK-dimension 4,
which should give us a better sense of how connected Hopf algebras
of higher GK-dimension look like. Here is the main result.

\begin{theorem}
\label{zzthm0.3} Let $H$ be a connected Hopf algebra of GK-dimension
4. Then $H$ is isomorphic to one of following.
\begin{enumerate}
\item[(a)]
Enveloping algebra $U({\mathfrak g})$ over a Lie algebra ${\mathfrak
g}$ of dimension $4$.
\item[(b)]
Enveloping algebra $U(L)$ over an anti-cocommutative coassociative
Lie algebra $L$ of dimension $4$.
\item[(c)]
Primitively-thin Hopf algebras of GK-dimension 4.
\end{enumerate}
\end{theorem}

Here are some details about Theorem \ref{zzthm0.3}.

\begin{remark}
\label{zzrem0.4} Let $H$ be as in Theorem \ref{zzthm0.3} and parts
(a,b,c) here match up with that of Theorem \ref{zzthm0.3}.
\begin{enumerate}
\item
All 4-dimensional Lie algebras over the complex numbers ${\mathbb
C}$ are listed in the book \cite[Theorem 1.1(iv), page 209]{OV}. In
this case, the lantern ${\mathfrak L}(H)$ is the unique graded Lie
algebra of dimension 4 generated by four elements in degree 1,
namely, the abelian Lie algebra of dimension 4.
\item
Anti-cocommutative coassociative Lie algebras of dimension $4$ are
classified in Theorem \ref{zzthm3.5}. Therefore Hopf algebras in
Theorem \ref{zzthm0.3}(b) are completely described. In this case,
${\mathfrak L}(H)$ is, up to isomorphism, the unique graded Lie
algebra of dimension 4 generated by three elements in degree 1,
namely, the Lie algebra ${\mathfrak h}_3\oplus k$ where ${\mathfrak
h}_3$ is the 3-dimensional Heisenberg Lie algebra.
\item
There are exactly four families of primitively-thin Hopf algebras of
GK-dimension 4, each of which is constructed explicitly in Section
\ref{zzsec4}, see Theorem \ref{zzthm4.23}. In this case, ${\mathfrak
L}(H)$ is, up to isomorphism, the unique graded Lie algebra of
dimension 4 generated by two elements in degree 1.
\end{enumerate}
\end{remark}

Some ideas can be extended to connected Hopf algebras of
GK-dimension 5. For example, the classification of the lantern of
$H$ is given in Remark \ref{zzrem2.10}. In higher GK-dimension, we
have the following result. Let $p(H)$ denote the dimension of the
space of all primitive elements. Let $P_2(H)$ be the space generated
by all anti-cocommutative elements of $H$.

\begin{theorem}
\label{zzthm0.5} Suppose $H$ is a connected Hopf algebra. If
$p(H)=\GKdim H-1<\infty$, then $H$ is isomorphic to the enveloping
algebra over an anti-cocommutative coassociative Lie algebra
$P_2(H)$. In this case $H$ is completely determined by $P_2(H)$.
\end{theorem}

\subsection*{Acknowledgments}
The authors thank Nicol{\'a}s Andruskiewitsch and Ken Brown for
their valuable suggestions and several conversations during the
Banff workshop in October 2012. D.-G. Wang was supported by the
National Natural Science Foundation of China (No. 10671016 and
11171183) and the Shandong Provincial Natural Science Foundation
of China (No. ZR2011AM013). J.J. Zhang and G. Zhuang were supported
by the US National Science Foundation (NSF grant No. DMS 0855743).

\section{Preliminaries}
\label{zzsec1}

Throughout let $k$ denote a base field. All vector spaces, algebras,
coalgebras are over $k$. For any coalgebra $C$, we use $\Delta$ and
$\epsilon$ for comultiplication and counit, respectively. We denote
the kernel of the counit by $C^+$. The {\it coradical}  $C_0$ of $C$
is defined to be the sum of all simple subcoalgebras of $C$. The
coalgebra $C$ is called {\it pointed} if every simple subcoalgebra
is one-dimensional, and is called {\it connected} if $C_0$ is
one-dimensional. Also, we use $\{C_n\}_{n=0}^{\infty}$ to denote the
coradical filtration of $C$ \cite[5.2.1]{Mo}.

For a pointed Hopf algebra $H$, the coradical filtration
$\{H_n\}_{n=0}^{\infty}$ is a Hopf algebra filtration \cite[p. 62]{Mo}.
As a consequence, the associated graded algebra is also a Hopf algebra,
which we denote by $\gr H$. Also, we use $\gr H(n)$ to denote the $n$-th 
homogeneous component of $\gr H$ (i.e. $\gr H(n)=H_n/H_{n-1}$).

Let $\LA$ be the category of Lie algebras and let $\HA_{cc}$ be the
category of connected cocommutative Hopf algebras. Then the first
assertion of the following is a consequence of
Milnor-Moore-Cartier-Kostant Theoerem \cite[Theorem 5.6.5]{Mo}. The
second assertion is a well-known fact in ring theory.

\begin{proposition}
\label{zzpro1.1} Let ${\text{char}}\; k=0$. Then the assignment
${\mathfrak g}\to U(\mathfrak g)$ defines an equivalence between
categories $\LA$ and $\HA_{cc}$. If $\dim \mathfrak g<\infty$, then
$$\GKdim U(\mathfrak g)=\gldim U(\mathfrak g)=\dim \mathfrak g.$$
\end{proposition}

In some sense, noncocommutative connected Hopf algebras are a
generalization of the universal enveloping algebra over a Lie algebra.

Let $H$ be a connected Hopf algebra. By \cite[Proposition 6.4 and
Theorem 6.6]{Zh2}, $\gr H$ is a commutative domain. Suppose $\gr H$
is locally finite. Then the graded dual $(\gr H)^*$ is a Hopf algebra.

\begin{definition}
\label{zzdef1.2} Let $H$ be a connected Hopf algebra.
\begin{enumerate}
\item
Let $P(H)$ be the space of primitive elements in $H$ and let $p(H)$
be the dimension of $P(H)$.
\item
We say $H$ is {\it locally finite} if $p(H)<\infty$, or
equivalently, $H_i$ in the coradical filtration of $H$ is finite
dimensional for all $i$.
\item
$H$ is called {\it primitively-thin}, if $p(H)=2$.
\item
Suppose $H$ is locally finite. The {\it lantern} of $H$ is the
graded Lie algebra ${\mathfrak L}(H)$ such that $U(\mathfrak
L(H))\cong (\gr H)^*$. In other words, ${\mathfrak L}(H)=P((\gr
H)^*)$.
\end{enumerate}
\end{definition}

The following lemma is easy. Part (a) of the following lemma says
that ${\mathfrak L}(H)$ is a kind of the abelianization of $H$. Parts
(e,f) justify calling $H$ primitively-thin when $p(H)=2$.

\begin{lemma}
\label{zzlem1.3} Let $H$ be a locally finite connected Hopf algebra.
\begin{enumerate}
\item
If $H$ is the (universal) enveloping algebra $U({\mathfrak g})$ for
a finite dimensional Lie algebra ${\mathfrak g}$, then ${\mathfrak
L}(H)$ is the abelian Lie algebra of dimension equal to $\dim
{\mathfrak g}$.
\item
The ${\mathfrak L}(H)$ is a positively graded Lie algebra generated
in degree 1 and $\dim {\mathfrak L}(H)=\GKdim H$.
\item
${\mathfrak L}(H)_1=(\gr H)^*_1 =(H_1/k)^*=P(H)^*$.
\item
$p(H)=1$ if and only if $H=k[x]$.
\item
\cite[Lemma 5.11]{Zh1}
If $\GKdim H\geq 2$, then $p(H)\geq 2$.
\item
$H$ is primitively-thin if and only if ${\mathfrak L}(H)$ is a
graded Lie algebra generated by two elements in degree 1.
\end{enumerate}
\end{lemma}

\begin{proof} (a) In this case $\gr H=U(A)$ where $A$ is an
abelian Lie algebra with $\dim A=\dim {\mathfrak g}$.
So $\gr H$ is a commutative and cocommutative Hopf
algebra and $(\gr H)^*\cong \gr H$ as Hopf algebras.
Consequently, $(\gr H)^*\cong U(A)$. Thus ${\mathfrak L}(H)\cong A$.

(b) Since $\gr H$ is coradically graded, $(\gr H)^*$ is generated in
degree 1 as an algebra. Since $(\gr H)^*$ is cocommutative
\cite[Proposition 6.4]{Zh2} and ${\rm{char}}\; k=0$,
$(\gr H)^*\cong U({\mathfrak L})$ and ${\mathfrak L}$ is a graded Lie
algebra generated in degree 1. Finally,
$$\GKdim H=\GKdim \gr H=\GKdim (\gr H)^*=\dim {\mathfrak L}(H).$$

(c) It follows from definition and part (b).

(d) If $p(H)=1$, then ${\mathfrak L}(H)_1$ is 1-dimensional. By part
(b), ${\mathfrak L}(H)$ is generated by ${\mathfrak L}(H)_1$. Thus
${\mathfrak L}(H)={\mathfrak L}(H)_1$, which is 1-dimensional.

(e,f) These follow from part (b) and definition.
\end{proof}

Given any graded Lie algebra finitely generated in degree 1, say
$\mathfrak g$, $U({\mathfrak g})$ is a locally finite graded Hopf
algebra. Since $U({\mathfrak g})$ is isomorphic to $(U({\mathfrak
g})^*)^*$, the lantern of the Hopf algebra $H=(U({\mathfrak g}))^*$
is $\mathfrak g$. Therefore every graded Lie algebra finitely
generated in degree 1 appears as the lantern of some connected Hopf
algebras.

\begin{lemma}
\label{zzlem1.4} Let $H$ be a connected Hopf algebra of GK-dimension
4 and let ${\mathfrak L}$ be the lantern of $H$. Then ${\mathfrak
L}$ is isomorphic to one of the following:
\begin{enumerate}
\item
the abelian Lie algebra of dimension 4 concentrated in degree 1.
\item
the graded Lie algebra of dimension 4 with a basis $\{a,b,c,
[a,b]\}$ where $a,b,c,$ are in degree 1 and $[a,b]$ is in degree 2,
subject to the relations $[c,{\mathfrak L}]=0=[[a,b],{\mathfrak
L}]$. This Lie algebra is isomorphic to the Lie algebra ${\mathfrak
h}_3\oplus k$ where ${\mathfrak h}_3$ is the 3-dimensional
Heisenberg Lie algebra.
\item
the graded Lie algebra of dimension 4 with a basis $\{a, b, [a,b],
[[a,b],b]\}$ where $a,b$ are in degree 1, $[a,b]$ is in degree 2 and
$[[a,b],b]$ is in degree 3, and subject to the relations
$[[a,b],a]=0= [[[a,b],b], {\mathfrak L}]=0$.
\end{enumerate}
\end{lemma}

\begin{proof} By Lemma \ref{zzlem1.3}(b), ${\mathfrak L}=
\oplus_{i\geq 1} {\mathfrak L}_i$ is a graded Lie algebra generated
in degree 1 of dimension 4. By Lemma \ref{zzlem1.3}(c,e), the degree
1 component ${\mathfrak L}_1$ has dimension either 2 or 3 or 4.

If $\dim {\mathfrak L}_1=4$, then ${\mathfrak L}={\mathfrak L}_1$
which must be abelian. This is case (a).

If $\dim {\mathfrak L}_1=3$, pick a basis, say $\{a,b,c\}$ of
${\mathfrak L}_1$. Then $\dim {\mathfrak L}_2=1$ and ${\mathfrak
L}_i=0$ for all $i>2$. Let $z$ be a basis of ${\mathfrak L}_2$. By
linear algebra, up to a basis change, $z=[a,b]$ and $c$ is in the
center of ${\mathfrak L}$. This is case (b).

If $\dim {\mathfrak L}_1=2$, pick a basis, say $\{a,b\}$ of
${\mathfrak L}_1$. Then $\dim {\mathfrak L}_2=1$ with a basis
$[a,b]$, $\dim {\mathfrak L}_3=1$ and $\dim {\mathfrak L}_i=0$ for all $i>3$.
Then $[[a,b],a]$ and $[[a,b],b]$ are linearly dependent. Up to a
base change we may assume that $[[a,b],a]=0$. Thus $[[a,b],b]$ is
the fourth basis element of ${\mathfrak L}$. This is case (c).
\end{proof}

The three different cases in Lemma \ref{zzlem1.4} will be a
guideline for our classification.

Later in this paper we will use the cohomology of coalgebras as a
tool, which we briefly recall here. Let $k\subset C$ be a connected
coalgebra. Then $C^+$, the kernel of the counit, becomes a coalgebra
without counit by setting
$$\delta(x)=\Delta(x)-(1\otimes x+x\otimes1)$$
Let $\Omega C$ be the tensor algebra $TC^+$ with a differential
determined by
$$\partial(x)=\delta(x)$$
for all $x\in C^+$ with given degree one. Extend $\partial$ to the
algebra $\Omega(C)$ as a derivation. Since $\delta$ is
coassociative, $\partial$ is a differential (i.e., $\partial^2=0$)
and $(\Omega C,\partial)$ is a dga. We call $(\Omega C,\partial)$
the cobar construction of $C$. The $i$th coalgebra cohomology of $C$
is defined to be the $i$th cohomology of $(\Omega C,\partial)$, namely
$H^i(\Omega C)$, for any integer $i$.

\section{Coassociative Lie algebras and their enveloping algebras}
\label{zzsec2}

We first recall the definition of coassociative Lie algebras.

\begin{definition}\cite[Definition 1.1]{WZZ3}
\label{zzdef2.1}
A Lie algebra $(L, [\;,\;])$ together with a coproduct
$\delta: L\to L\otimes L$ is called a {\it coassociative Lie algebra}
(or {\it CLA}, for short) if
\begin{enumerate}
\item
$(L,\delta)$ is a coalgebra, namely, $\delta$ is coassociative without
counit,
\item
$\delta$ and $[\;,\;]$ satisfies the following condition in the usual
  enveloping algebra $U(L)$ of the Lie algebra $L$,
\begin{equation}
\label{E2.1.1}\tag{E2.1.1}
\delta([a,b])= b_1\otimes [a, b_2]+[a,b_1]\otimes b_2+
[a_1,b]\otimes a_2+a_1\otimes [a_2,b]+[\delta(a), \delta(b)]
\end{equation}
for all $a,b\in L$. Here $\delta(x)=x_1\otimes x_2$ following Sweedler
with $\Sigma$ omitted.
\end{enumerate}
\end{definition}

Let $\LDA$ be the category of coassociative Lie algebras (CLAs). Some
basic properties of CLAs can be found in
\cite{WZZ3}. The enveloping algebra of a CLA
is defined as follows.

\begin{definition}\cite[Definition 1.8]{WZZ3}
\label{zzdef2.2} Let $L$ be a CLA. The enveloping algebra of $L$,
denoted by $U(L)$, is defined to be the bialgebra, whose algebra
structure is equal to the enveloping algebra of the Lie algebra $L$
without $\delta$, namely, $U(L)=k\langle L\rangle / (ab-ba=[a,b],
\forall \; a, b\in L)$, and whose coalgebra structure is determined
by
$$\Delta(a)= a\otimes 1+1\otimes a+\delta(a), \quad \epsilon(a)=0$$
for all $a\in L$.
\end{definition}

In general $U(L)$ is not a Hopf algebra, and it is a Hopf algebra if
and only if $L$ is locally conilpotent
\cite[Theorem 0.1, Definition 1.10]{WZZ3}.

\begin{definition}
\label{zzdef2.3} Let $L_1$ and $L_2$ be  CLAs.
\begin{enumerate}
\item
We say $L_1$ and $L_2$ are {\it quasi-equivalent}, and denoted
by $L_1\sim L_2$, if $U(L_1)$ is isomorphic to $U(L_2)$ as bialgebras.
\item
A coassociative coalgebra $(L,\delta)$ is called {\it anti-cocommutative}
if $\tau\delta=-\delta$ where the flip $\tau:L^{\otimes 2}\to L^{\otimes 2}$
is defined by $\tau(a\otimes b)=b\otimes a$.
\end{enumerate}
\end{definition}

For an anti-cocommutative CLA $L$, the enveloping
algebra $U(L)$ is a connected Hopf algebra since $L$ is conilpotent by
\cite[Lemma 2.8(b)]{WZZ3}.

In rest of this section we assume that ${\text{char}}\; k\neq 2$.
Let $H$ be a general Hopf algebra and let $P(H)$ denote the $k$-subspace
of $H$ consisting of all primitive elements in $H$. It is well known
that $P(H)$ is Lie algebra. The dimension of $P(H)$ is denoted by $p(H)$.

Let $\delta_H: H\to H^{\otimes 2}$ be the map defined by
\begin{equation}
\label{E2.3.1}\tag{E2.3.1}
\delta_H(h)=\Delta(h)-(h\otimes 1+1\otimes h)
\end{equation}
for all $h\in H$.

\begin{definition}
\label{zzdef2.4}
Let $H$ be a Hopf algebra.
\begin{enumerate}
\item
An element $f\in H^{\otimes 2}$ is called {\it symmetric}
if $\tau(f)=f$. An element $f\in H^{\otimes 2}$ is called
{\it skew-symmetric} if $\tau(f)=-f$.
\item
Define
$$P_2(H)=\{ x\in H\mid {\text{$\delta_H(x)$ is skew-symmetric and
lies in $P(H)^{\otimes 2}$}}\}.$$
The dimension of $P_2(H)$ is denoted by $p_2(H)$.
\item
The anti-cocommutative-space of $H$ is the quoteient space
$P_2(H)/P(H)$, denoted by $P_2'(H)$. The dimension of
$P_2'(H)$ is denoted by $p_2'(H)$.
\end{enumerate}
\end{definition}

Here is a list of basic properties of $P_2(H)$. Part (a) justifies
Definition \ref{zzdef2.4}(b,c).

\begin{lemma}
\label{zzlem2.5}
Let $H$ be a Hopf algebra. Let $\delta=\delta_H$ and
$[\;,\;]=[\;,\;]_H$.
\begin{enumerate}
\item
Every anti-cocommutative subcoalgebra of $(H,\delta)$ is contained
in $P_2(H)$. Consequently, $(P_2(H),\delta)$ is the largest
anti-cocommutative subcoalgebra of $(H,\delta)$.
\item
$P_2(H)=\{x\in H\mid \delta(x)\subset P_2(H)^{\otimes 2},
\tau\delta(x)=-\delta(x)\}$.
\item
$P_2(H)$ is Lie subalgebra of $H$ if and only if $[\delta(x),\delta(y)]=0$
for all $x,y\in P_2(H)$.
\item
Suppose $P(H)$ is abelian. Then $(P_2(H),[\;,\;])$ is a Lie subalgebra
of $H$ and $[P_2(H),P_2(H)]\subset P(H)$.
\item
If $P_2(H)/P(H)$ is 1-dimensional, then $P_2(H)$ is a Lie subalgebra
of $H$.
\item
$P_2(H)$ is a Lie module over $P(H)$.
\item
$p_2'(H)\leq {p(H)\choose 2}$.
\end{enumerate}
\end{lemma}

\begin{proof} Clearly $P(H)\subset P_2(H)$. By definition,
$\delta(P_2(H))\subset P(H)^{\otimes 2}\subset P_2(H)^{\otimes 2}$
and $\tau \delta(x)=-\delta(x)$ for all $x\in P_2(H)$. Hence
$P_2(H)$ is an anti-cocommutative subcoalgebra of $(H,\delta)$.

(a) By definition, $\ker \delta=P(H)$.
Let $C$ be any anti-cocommutative subcoalgebra of
$(H,\delta)$. By \cite[Lemma 2.8(b)]{WZZ3},
$$\delta(C)\subset (\ker \delta)^{\otimes 2}\subset
P(H)^{\otimes 2}.$$
Since $C$ is anti-cocommutative, $C\subset P_2(H)$ by definition.

(b) Let $C=\{x\in H\mid \delta(x)\subset P_2(H)^{\otimes 2},
\tau\delta(x)=-\delta(x)\}$. Then, by definition, $P_2(H)\subset C$.
So $\delta(C)\subset C^{\otimes 2}$ and $C$ is an anti-cocommutative
subcoalgebra of $(H,\delta)$. The assertion now follows from part (a).

(c) For any $x,y\in H$,
$$\begin{aligned}
\delta([x,y])&=\Delta([x,y])-[x,y]\otimes 1-1\otimes [x,y]\\
&=[\Delta(x),\Delta(y)]-[x,y]\otimes 1-1\otimes [x,y]\\
&=[\delta(x)+x\otimes 1+1\otimes x, \delta(y)+y\otimes 1+1\otimes y]
-[x,y]\otimes 1 -1\otimes [x,y]\\
&=w(x,y)+v(x,y),
\end{aligned}$$
where
$$w(x,y)=[\delta(x),y\otimes 1+1\otimes y]+[x\otimes 1+1\otimes x,
\delta(y)]$$
and
$$v(x,y)=[\delta(x),\delta(y)].$$
Now let $x,y\in P_2(H)$. By definition, $\delta(x), \delta(y)\in
P(H)^{\otimes 2}$. In this case $w(x,y)\in P(H)^{\otimes 2}$.
Since $\delta(x)$ and $\delta(y)$ are skew-symmetric, so is $w(x,y)$.
But $v(x,y)$ is symmetric. Hence, $\delta([x,y])$ is skew-symmetric
if and only if $v(x,y)=0$. The assertion follows.

(d) By part (c), $P_2(H)$ is a Lie subalgebra of $(H,[\;,\;])$.
For $x,y\in P_2(H)$, $\delta([x,y])=w(x,y)+v(x,y)$ where $w(x,y),v(x,y)$
are defined as in the proof of part (c). Since $P(H)$ is abelian, both
$w(x,y)$ and $v(x,y)$ are 0. Hence $\delta([x,y])=0$ and consequently,
$[x,y]\in P(H)$. The assertion follows.

(e) Since $P_2(H)/P(H)$ is 1-dimensional, $P_2(H)=kf\oplus P(H)$ for some
$f\in P_2(H)$. For any $x,y\in P_2(H)$, write $x=af+x_0$ and $y=bf+y_0$
for some $a,b\in k$ and $x_0,y_0\in P(H)$. Then $\delta(x)=a\delta(f)$
and $\delta(y)=b\delta(f)$. Hence $[\delta(x),\delta(y)]
=[a\delta(f),b\delta(f)]=0$. The assertion follows from part (c).

(f)  Let $x\in P_2(H)$ and $y\in P(H)$. Then $\delta([x,y])=
[\delta(x),1\otimes y+y\otimes 1]\in P(H)^{\otimes 2}$.
Since $\delta(x)$ is skew-symmetric, so is $[\delta(x),
1\otimes y+y\otimes 1]$. Hence $[x,y]\in P_2(H)$.

(g)  This follows from the fact that $\delta$ defines a $k$-linear
injective map from $P_2(H)/P(H)\to P(H)\wedge P(H)$.
\end{proof}

\begin{lemma}
\label{zzlem2.6} Let $H$ be a Hopf algebra. Suppose that
${\text{char}}\; k =0$. In parts (c,d,e), assume that $H$ is a
connected Hopf algebra.
\begin{enumerate}
\item
$p_2(H)\leq \GKdim H$.
\item
Let ${\mathfrak g}$ be a Lie algebra. Then
$P_2(U({\mathfrak g}))=P(U({\mathfrak g}))={\mathfrak g}$.
\item
Let $U$ be the Hopf subalgebra of $H$ generated by $P(H)$. If $U\neq
H$, then $P(H)\neq P_2(H)$ and $\GKdim U<\GKdim H$.
\item
$P(H)\cong \gr H(1)=P(\gr H)$.
\item
$P_2(H)\cong P_2(\gr H)$ and $P_2(\gr H)\oplus \gr
H(1)^2=\gr
H(1)\oplus \gr H(2)$.
\end{enumerate}
\end{lemma}

\begin{proof}
(a) The subcoalgebra $P_2(H)+k1$ is connected and counital by
\cite[Lemma 2.4]{WZZ3}. Then the subbialgebra of $H$
generated by $P_2(H)+k1$ is connected, and whence a connected
Hopf algebra \cite[Lemma 5.2.1]{Mo}. Therefore, after replacing it
by the Hopf subalgebra generated by $P_2(H)+k1$, we may assume that
$H$ is a connected Hopf algebra.

By \cite[Theorem 6.10]{Zh2}, $\gr H\cong k[x_1,\cdots,x_n]$ where
$n=\GKdim H$. Arranging $\{x_i\}$ so that $\deg x_i=1$ for all $i=1,
\cdots, p_1$ and $\deg x_i=2$ for all $i=p_1+1,\cdots, p_2$ and
$\deg x_i>2$ for all $i>p_2$. Then $\{x_1,\cdots,x_{p_1}\}$ is a
basis of $\gr H(1)=H_1/H_0\cong P(H)$, and $\{x_{p_1+1},
\cdots,x_{p_2}\}$ is a basis of $\gr H(2)/\gr H(1)^2$. Since 
$\gr H(2)=H_2/H_1$ and $\gr H(1)=H_1/H_0$, $\gr H(2)/\gr H(1)^2\cong
H_2/H_1^2$. Note that $P_2(H)$ is a subspace of $H_2$ by the
definition of $P_2(H)$. For any $x\in P_2(H)\setminus P(H)$,
$\delta(x)$ is nonzero and skew-symmetric by the definition of
$P_2(H)$ and $P(H)$, and for every $y\in H_1^2$, an easy calculation
shows that $\delta(y)$ is symmetric. Let $C$ be a subspace of
$P_2(H)$ such that $P_2(H)=C\oplus P(H)$. Then the above discussion
says that $C\subset H_2$ and $C\cap H_1^2=\{0\}$. Therefore
$$\dim P_2(H)/P(H)=\dim C\leq \dim H_2/H_1^2=\dim \gr H(2)/\gr H(1)^2
=p_2-p_1.$$
Thus
$$p_2(H)=\dim P_2(H)=\dim C+\dim P(H)\leq (p_2-p_1)+p_1=p_2\leq n=\GKdim H.$$

(b) Follows by a direct computation.

(c) Since $U\neq H$, by \cite[Lemma 7.4]{Zh2}, $\GKdim U<\GKdim H$ if
$\GKdim U<\infty$.

Since $U$ is the Hopf subalgebra of $H$ generated by $P(H)$, $U$ is
cocommutative and, whence, $U$ is the   enveloping algebra
$U({\mathfrak g})$ for some Lie algebra ${\mathfrak g}$. Clearly,
${\mathfrak g}=P(U)=P(H)$.

Consider the coradical filtrations $\{U_i\}_{i\in {\mathbb N}}$ and
$\{H_i\}_{i\in {\mathbb N}}$ of $U$ and $H$ respectively. Since
$U\subsetneq H$, there is a minimal $n$ such that $U_n\subsetneq
H_n$.  The equality $P(U)=P(H)$ implies that $n\geq 2$. Pick any
$f\in H_n\setminus U_n$, $\delta(f) \in H_{n-1}^{\otimes
2}=U_{n-1}^{\otimes 2}$. This means that $\delta(f)\in U^{\otimes
2}$ is a coalgebra 2-cocylce of $U$.

Since $U$ is the   enveloping algebra $U(\mathfrak g)$, $U$, as a
coalgebra, is isomorphic to $U(V)$ where $V$ is the abelian Lie
algebra of dimension equal to $\dim {\mathfrak g}$. So one can
forget about the Lie structure of ${\mathfrak g}$ when computing the
coalegbra cohomology of $U$. A cohomology computation shows that any
2-cocylce is congruent to an element in ${\mathfrak g}\wedge
{\mathfrak g}$ modulo some 2-coboundary. This fact is a
special case of a more general result in \cite[Proposition
4.1]{WZZ4}. This means that there is an element $g\in U$ such that
$\delta(f-g)\in {\mathfrak g}\wedge {\mathfrak g}$. Since $f\not\in
U$, $f-g\in H\setminus U$, or $f-g\in P_2(H)$ as $\delta(f-g)\in
{\mathfrak g}\wedge {\mathfrak g}$ is skew-symmetric. The assertion
follows.

(d) By the definition of coradical filtration, $H_1=k1\oplus P(H)=H_0\oplus
P(H)$. Hence $P(H)\cong \gr H(1)$. Since $\gr H$ is coradically
graded \cite[Remark 2.2]{Zh2}, $P(\gr H)=\gr H(1)$.

(e) Let $x\in P_2(H)$ and let $a$ be the associated element in $\gr
H$. If $x\in P(H)$, then $a\in P(\gr H)$ by part (d) and the map
$x\to a$ is an isomorphism when restricted to $P(H)$. Now suppose
$x\not\in P(H)$. Then $a\not\in P(\gr H)$ by part (d). Since
$\delta(x)\in P(H)^{\otimes 2}$ is skew-symmetric, so is
$\delta(a)\in P(\gr H)^{\otimes 2}$. Thus we have an injective map
$P_2(H)\to P_2(\gr H)$. Now let $a\in P_2(\gr H)$. Since $P(H)\cong
\gr H(1)$, we may assume that $a\in \gr H(2)$. Let $\{x_i\}$ be a
basis of $P(H)$ and $\{y_i\}$ be the corresponding basis of $\gr H(1)$.
 Then $\delta(a)=\sum_{i,j}c_{ij}(y_i\otimes y_j-y_j\otimes
y_i)$ for some $c_{ij}\in k$. Let $x\in H_2^+$ be a preimage of $a$.
Then $\delta(x)=\sum_{i,j}c_{ij}(x_i\otimes x_j-x_j\otimes x_i)+
1\otimes z_1+z_2\otimes 1+c\otimes 1$ where $z_1, z_2\in P(H)$ and
$c\in k$. The counit axiom for $\Delta$ implies that $z_1=z_2=c=0$.
Therefore $x\in P_2(H)$. The first assertion follows.

For the second assertion we note that
$$\gr H(1)=P(\gr H)\subset P_2(\gr H)\subset\gr H(1)\oplus \gr H(2)$$ and that
$$P_2(\gr H)\cap \gr H(1)^2=\{0\}.$$
It remains to show that $\gr H(2) \subset P_2(\gr H)\oplus \gr H(1)^2$. 
By replacing $H$ with $\gr H$, we may assume that $H$ is
coradically graded. For any $x\in H_2^+$, $\delta(x)=\sum_{i,j}
c_{ij} x_i\otimes x_j+ 1\otimes z_1+z_2\otimes 1+c\otimes 1$ where
$x_i, z_1, z_2\in P(H)$ and $c\in k$. The counit axiom for $\Delta$
implies that $z_1=z_2=c=0$. Replacing $x$ by $x-\sum_i \frac{1}{2}
c_{ii}x_i^2-\frac{1}{2}\sum_{i<j}(c_{ij}+c_{ji}) x_ix_j$,
$\delta(x)$ becomes skew-symmetric. So $x\in P_2(H)$. The second
assertion follows.
\end{proof}

Here is our main result of this section.

\begin{theorem}
\label{zzthm2.7} Suppose that ${\text{char}}\; k=0$ and $H$ is a
connected Hopf algebra. If
$$\GKdim H\leq p(H)+1<\infty,$$
then $H\cong U(L)$ for some anti-cocommutative CLA $L=P_2(H)$. If,
further,
$\GKdim H=p(H)$, then $H=U({\mathfrak g})$ where ${\mathfrak g}=P(H)$.
\end{theorem}

\begin{proof} Let ${\mathfrak g}$ be the Lie algebra $P(H)$ of the
primitive elements of $H$. If $H=U({\mathfrak g})$, the assertion is
trivial as we have a natural embedding $\LA\to \LDA$.

For the rest of the proof we assume that $H\neq U({\mathfrak g})$.
By Lemma \ref{zzlem2.6}(c), $P_2(H)\neq P(H)$. Since $P(H)$ is a
proper subspace of $P_2(H)$, $p(H)+1\leq p_2(H)$. By Lemma
\ref{zzlem2.5}(h), $p_2(H)\leq \GKdim H$. By hypothesis, $\GKdim
H\leq p(H)+1<\infty$. Therefore there is only one possibility,
namely, $p_2(H)=p(H)+1=\GKdim H$. By Lemma \ref{zzlem2.5}(e)
$P_2(H)$ is a Lie subalgebra of $H$. By Lemma \ref{zzlem2.5}(a),
$(P_2(H),\delta)$ is an anti-cocommutative subcoalgebra of
$(H,\delta)$.

Let $K$ be the subbialgebra of $H$ generated as an algebra by the
connected subcoalgebra $(P_2(H)+k1, \Delta)$. Then $K$ is a Hopf
subalgebra of $H$ as mentioned in the proof of Lemma
\ref{zzlem2.6}(a). By Lemma \ref{zzlem2.5}(a), $P_2(K)\supseteq
P_2(H)$ and clearly $P_2(H)\supseteq P_2(K)$, we have
$P_2(K)=P_2(H)$. By Lemma \ref{zzlem2.6}(b), $P_2(U({\mathfrak
g}))={\mathfrak g}$, and by assumption in the previous paragraph and
Lemma \ref{zzlem2.6}(c), $P_2(H)\supsetneq {\mathfrak g}$, we
conclude that $U({\mathfrak g})$ is a proper subalgebra of $K$.
By \cite[Lemma 7.4]{Zh2},
$$\GKdim K\geq \GKdim U({\mathfrak g})+1=p(H)+1.$$
By hypothesis, $\GKdim H\leq p(H)+1$ and obviously $\GKdim H\geq \GKdim
K$. Hence $\GKdim H=\GKdim K=p(H)+1=p_2(H)$. By \cite[Lemma 7.4]{Zh2},
$H=K$.

Next we show that $L:=P_2(H)$ is a CLA. By the
second paragraph, $L$ is both a Lie algebra and a coalgebra. It
remains to verify \eqref{E2.1.1}. Choose a basis of ${\mathfrak g}$,
say,  $\{x_i\}$, and  an element $z\in L\setminus {\mathfrak g}$.
Then $\{x_i\}\cup \{z\}$ is a basis of $L$. It is trivial that
\eqref{E2.1.1} holds for $(a,b)=(z,z)$ and for $(a,b)=(x_i,x_j)$
since $\delta(x_i)=0$. It remains to show \eqref{E2.1.1} for
$(a,b)=(z,x_i)$ (and by symmetry for $(a,b)=(x_i,z)$). Since
$\delta(x_i)=0$, we have
$$\begin{aligned}
\delta([z,x_i])&=[\delta(z), x_i\otimes 1+1\otimes x_i]\\
&=[z\otimes 1+1\otimes z, \delta(x_i)]+[\delta(z), x_i\otimes
1+1\otimes x_i]+[\delta(z),\delta(x_i)]
\end{aligned}
$$
which holds in $H^{\otimes 2}$ and hence holds in $L^{\otimes 2}$.
Thus the above holds in $U(L)^{\otimes 2}$, which verifies \eqref{E2.1.1}.

Let $U(L)$ be the enveloping algebra of the CLA $L$. There is a
canonical Hopf algebra map $\phi:U(L) \to H$
sending $x\in L$ to $x$. By \cite[Theorem 5.3.1]{Mo}, the map $\phi$
is injective since the restriction of $\phi$ on the space of
primitive elements, namely, $\phi\mid_{P(H)}$, is injective. Now by
\cite[Lemma 7.4]{Zh2} again this is an isomorphism since $\GKdim
H=\GKdim U(L)$. Finally by the definition of $P_2(H)$, $L$ is
anti-cocommutative.
\end{proof}

\begin{proposition}
\label{zzpro2.8} Suppose ${\text{char}}\; k=0$. Let $L$ be an
anti-cocommutative CLA.
\begin{enumerate}
\item
$P_2(U(L))=L$.
\item
Let $P=\ker (\delta: L\to L^{\otimes 2})$. Then
$\gr U(L)=k[P\oplus L/P]$ where elements of $P$ are in degree 1 and
that of $L/P$ are in degree 2.
\item
The coradical of $U(L)$ is given by $U(L)_0=k1$,
$U(L)_1=P+k1$, $U(L)_2=(P+k1)^2+L$, and, for $n\geq 3$,
$$U(L)_n=\sum_{i=1}^{n-1} U(L)_{i}\cdot U(L)_{n-i}.$$
\item
Suppose $L$ is finite dimensional. Then ${\mathfrak L}(U(L))\cong
P^*\oplus (L/P)^*$ where the Lie algebra structure of $P^*\oplus
(L/P)^*$ is induced by the coalgebra structure of $L$.
\end{enumerate}
\end{proposition}

\begin{proof} For simplicity, we assume that $\dim L<\infty$ in
the following proof. The assertion holds in general, but the proof
requires more computation which we omit here.

(a) By construction and Lemma \ref{zzlem2.5}(a), $L$ is a subspace
of $P_2(U(L))$. By Lemma \ref{zzlem2.6}(a),
$$\dim P_2(U(L))\leq \GKdim U(L)=\dim L\leq \dim P_2(U(L)).$$
Hence $L=P_2(U(L))$ as $L$ is finite dimensional.

(b) By the proof of Lemma \ref{zzlem2.6}(a), $k[P\oplus L/P]$ is a
Hopf subalgebra of $\gr U(L)$. Obviously they have the same
GK-dimension. Now the assertion follows from \cite[Lemma 7.4]{Zh2}.

(c) This follows from part (b) and the fact $\gr U(L)$ is
coradically graded.

(d) By part (b), $H:=\gr U(L)=k[P\oplus L/P]$ where $P\oplus L/P$ is
the minimal generating space of $H$.  Then $P^*\oplus (L/P)^*$ is a
minimal co-generating space of the dual $K:=H^*=(\gr U(L))^*$. Let
$\{x_i\}_{i}$ and $\{y_j\}_j$ be a basis of $P$ and $L/P$
respectively, and $\{x_i^*\}_{i}$ and $\{y_j^*\}_j$ be the dual
basis of $P^*$ and $(L/P)^*$ respectively. It follows from the
definition and the fact that $x_i$ and $y_j$ are generators that
both $x_i^*$ and $y_j^*$ are primitive elements. Therefore
$P^*\oplus (L/P)^*\subset {\mathfrak L}(U(L))$. Consequently,
$P^*\oplus (L/P)^*={\mathfrak L}(U(L))$ as they have the same
$k$-dimension. Suppose, for any $s$, $\delta(y_s)=\sum_{i<j}
c^{ij}_s(x_i\otimes x_j-x_j\otimes x_i)$ for some $c^{ij}_s\in k$,
and define $c^{ji}_s=-c^{ij}_s$. Then
$$\Delta_H(y_s)=y_s\otimes 1+1\otimes y_s+\sum_{i<j}c^{ij}_s
(x_i\otimes x_j-x_j\otimes x_i).$$ By the $k$-linear pairing
$K\times H\to k$, we have
$$\begin{aligned}
\; [x_i^*,x_j^*](y_t)&=(x_i^* x_j^* - x_j^* x_i^* ) (y_t)\\
&=(x_i^*\otimes x_j^*-x_j^*\otimes x_i^*)(\Delta_H(y_t))\\
&=(x_i^*\otimes x_j^*-x_j^*\otimes x_i^*)(y_t\otimes 1+1\otimes
y_t+\sum_{i'<j'}
c^{i'j'}_t(x_{i'}\otimes x_{j'}-x_{j'}\otimes x_{i'}))\\
&=2c^{ij}_t = \sum_{s}c^{ij}_{s} (2y_{s}^*)(y_t).\footnote{I changed the index.}
\end{aligned}
$$
Hence $[x_i^*,x_j^*]=\sum_{s}2c^{ij}_{s} y_{s}^*$ for all
$i,j$. Therefore the assertion follows since the coefficients $c^{ij}_{s}$ are determined by the coalgebra of $L$.
\end{proof}

\begin{proposition}
\label{zzpro2.9} Let $L$ be a conilopent CLA of dimension $\leq 4$.
Then it is quasi-equivalent to an anti-cocommutative CLA.
\end{proposition}

\begin{proof} By \cite[Lemma 2.4]{WZZ3}, $U(L)$ is a connected Hopf algebra.
By Lemma \ref{zzlem1.3}(e), $P(U(L))$ has dimension at least 2. If
$\dim L\leq 3$, the assertion follows from Theorem \ref{zzthm2.7}.
For the rest we consider the case when $\dim L=4$.

Let $P_n(L)=\ker \delta^n$. Since $L$ is conilpotent,
$P_{i-1}(L)\neq P_{i}(L)$ if $P_{i-1}(L)\neq L$.

If $\dim P_1(L)=4$, then $L$ is 4-dimensional Lie algebra with
trivial $\delta$-structure.

If $\dim P_1(L)=1$, then $\dim P_2(L)=2$ with an element $x_2\in
P_2(L)\setminus P_1(L)$ such that $\delta(x_2) =x_1\otimes x_1$. For
any $x_3\in P_3(L)\setminus P_2(L)$, write $\delta(x_3)=a x_1\otimes
x_2+b x_2\otimes x_1+c x_2\otimes x_2$. The coassociativity on $x_3$
implies that $a=b$ and $c=0$. Without loss of generality, we assume
that $\delta(x_3)=x_1\otimes x_2+x_2\otimes x_1$. Thus $x_1,
x_2-\frac{1}{2}x_1^2$ and $x_3-x_1x_2+\frac{1}{3}x_1^3$ are linearly
independent primitive elements in $U(L)$. Thus $\dim P(U(L))$ is at
least 3. Hence $L$ is quasi-equivalent to an anti-cocommutative
CLA by Theorem \ref{zzthm2.7}.

Suppose $\dim P_1(L)=2$ with a basis $\{x_1,x_2\}$. If $\dim P(U(L))
\geq 3$, $L$ is quasi-equivalent to an anti-commutative CLA by
Theorem \ref{zzthm2.7}. So we only need to consider the case when
$\dim P(U(L))=2$. If $\dim P_2(L) =4$, choose two basis elements
$y_1,y_2$ in $P_2(L)\setminus P_1(L)$. Then
$\delta(y_1-f(x_1,x_2))=a(x_1\otimes x_2-x_2\otimes x_1)$ and
$\delta(y_2-g(x_1,x_2))=b(x_1\otimes x_2-x_2\otimes x_1)$ for some
polynomial $f(x_1,x_2)$ and $g(x_1,x_2)$.  Since $P(U(L))=2$, both
$a$ and $b$ are nonzero. Thus $a(y_2-g)-b(y_1-f)$ is an extra
primitive element, a contradiction. Hence $\dim P_2(L)=3$ and $\dim
P_3(L)=4$. Let $x_3$ be a basis element in $P_2(L)$ and $x_4$ be a
basis element in $P_3(L)$. Write $\delta(x_4)=f\otimes
x_3+x_3\otimes g+a x_3\otimes x_3$ with $f,g\in P_1(L)$ and $a\in
k$. Coassociativity on $x_4$ implies that $a=0$, $f=g\neq 0$ and
$\delta(x_3) =b f\otimes f$ for some $b\in k$. Thus $L$ is symmetric
and hence quasi-equivalent to a Lie algebra by \cite[Corollary
2.6]{WZZ3}.

If $\dim P_1(L)=3$, then $\dim P(U(L))$ is at least $3$ and the
assertion follows from Theorem \ref{zzthm2.7}.
\end{proof}

Before we continue we would like to make a remark about the lantern
of connected Hopf algebras of GK-dimension 5, which could serve as a
guideline for the classification of connected Hopf algebras of
GK-dimension 5.

\begin{remark}
\label{zzrem2.10} Let $H$ be a connected Hopf algebra of
GK-dimension 5.
\begin{enumerate}
\item
Let $p(H)$ denote the dimension of the space of all primitive
elements. Then $p(H)$ is in the range $[2,5]$.
\item
If $p(H)=5$, then ${\mathfrak L}(H)$ is abelian (which is unique)
and $H$ is isomorphic to the enveloping algebra over a Lie algebra.
\item
If $p(H)=4$, then ${\mathfrak L}(H)$ is generated by four elements
in degree 1 (there are two such graded Lie algebras up to
isomorphism) and $H$ is isomorphic to the enveloping algebra over an
anti-cocommutative coassociative Lie algebra.
\item
If $p(H)=3$, then ${\mathfrak L}(H)$ is generated by three elements
in degree $1$ and there are two subcases.
\begin{enumerate}
\item[(d1)]
If the degree two component of ${\mathfrak L}(H)$ has dimension 2
(there is only one such graded Lie algebra), then $H$ is isomorphic
to an enveloping algebra over an anti-cocommutative coassociative
Lie algebra.
\item[(d2)]
If the degree two component of ${\mathfrak L}(H)$ has dimension 1
(there is only one such graded Lie algebra), then $H$ is not
isomorphic to the enveloping algebra over either an ordinary Lie
algebra or a coassociative Lie algebra, nor a primitively-thin Hopf
algebra.
\end{enumerate}
\item
If $p(H)=2$, then ${\mathfrak L}(H)$ is generated by two elements in
degree $1$ (there are two such graded Lie algebras), and $H$ is a
primitively-thin Hopf algebra.
\end{enumerate}
\end{remark}

\section{Classification of anti-cocommutative CLAs up to dimension 4}
\label{zzsec3}

In this section we classify all anti-cocommutative CLAs of
dimension up to 4. The case of dimension 1 is trivial.

\subsection{Dimension 2}
\label{zzsec3.1}
We start with an easy observation.

\begin{lemma}
\label{zzlem3.1} Let $L$ be an anti-cocommutative CLA of
dimension two. Then $\delta=0$ and $L$ is an ordinary Lie
algebra.
\end{lemma}

\begin{proof} Let $I=\ker \delta$. By \cite[Lemma 2.8(b)]{WZZ3},
$\delta(L)\subset I^{\otimes 2}$. If $I=0$, then $\delta=0$.
If $\dim I=1$, pick a nonzero element $x\in I$. Then, for every
$y\in L$, $\delta(y)= \lambda x\otimes x$ for some $\lambda \in k$.
By the anti-cocommutativity of $\delta$, $\lambda=0$. Thus
$\delta=0$. The remaining case is when $\dim I=2$, which implies
that $\delta=0$.
\end{proof}

It is well known that a 2-dimensional Lie algebra is either abelian
or solvable with $[x,y]=y$ for a suitable basis $\{x,y\}$.

\subsection{Dimension 3}
\label{zzsec3.2}
The following lemma is similar to
Lemma \ref{zzlem3.1}. The classification of 3-dimensional Lie
algebra is well-known, which is called Bianchi classification when
the base field is either the real numbers ${\mathbb R}$ or
the complex numbers ${\mathbb C}$ (for instance, see
\cite[p. 209, Theorem 1.1]{OV} if $k={\mathbb C}$). We will not
include the list here. Next we will consider those CLAs
with nontrivial $\delta$.

Let $a(\lambda_1,\lambda_2,\alpha)$ denote the CLA with a basis
$\{x,y,z\}$ whose Lie algebra structure on $L=:B(b)$ is determined
by
$$\begin{aligned}
\; [x,y]&=0,\\
[z,x]&=\lambda_1 x+\alpha y,\\
[z,y]&=\lambda_2 y\end{aligned}$$ for some
$\lambda_1,\lambda_2,\alpha\in k$, and whose coalgebra structure is
determined by $\delta(x)=\delta(y)=0$ and $\delta(z)=x\otimes
y-y\otimes x$.

\begin{lemma}
\label{zzlem3.2} Let $L$ be an anti-cocommutative CLA of
dimension 3 such that $\delta\neq 0$. Then $L$ is
isomorphic to one of the following:
\begin{enumerate}
\item
$a(0,0,0)$.
\item
$a(1,\lambda,0)$. And $a(1,\lambda',0)$ is isomorphic to
$a(1,\lambda,0)$ if and only if $\lambda'=\lambda$ or
$\lambda^{-1}$.
\item
$a(0,0,1)$.
\item
$a(1,1,1)$.
\item
The CLA $b(\lambda)$ with a basis $\{x,y,z\}$, whose Lie algebra
strcutrue is determined by
$$\begin{aligned}
\; [x,y]&=y,\\
[z,x]&=-z+\lambda y,\\
[z,y]&=0,
\end{aligned}$$
where $\lambda$ is in $k$, and whose coalgebra structure is
determined by $\delta(x)=\delta(y)=0$ and $\delta(z)=x\otimes
y-y\otimes x$. And $b(\lambda')$ is isomorphic to $b(\lambda)$ if
and only if $\lambda'=\lambda$.
\end{enumerate}
All CLAs listed above are pairwise non-isomorphic except for the
isomorphisms given in part (b).
\end{lemma}

\begin{proof} Let $I=\ker \delta$. Repeating the first part of the
proof of Lemma  \ref{zzlem3.1}, one sees that $\dim I >1$. Since
$\dim L=3$ and $\delta\neq 0$, we have $\dim I=2$. Pick any basis
$\{x,y\}$ of $I$. The only skew-symmetric elements in $I^{\otimes
2}$ are $\lambda (x\otimes y-y\otimes x)$ for some $\lambda\in
k^\times:=k\setminus \{0\}$. Let $z\in L\setminus I$ such that
$\delta(z)=x\otimes y-y\otimes x$.

Since $I$ is a Lie subalgebra of $L$ (see Lemma \ref{zzlem2.5}(e)),
we have the following two cases to consider.

Case 1: $I$ is abelian. Write
$$\begin{aligned}
\; [z,x]&=a_1 x+a_2 y+a_3 z,\\
[z,y]&=b_1 x+b_2 y+b_3 z.
\end{aligned}
$$
Applying \eqref{E2.1.1} to $(a,b)=(z,x)$, we have
$$a_3(x\otimes y-y\otimes x)=\delta([z,x])=
[\delta(z), x\otimes 1+1\otimes x]=0.$$
Hence $a_3=0$. By symmetry, $b_3=0$.

Using a linear transformation $f: x\to c_{11}x+c_{12}y, y\to
c_{21}x+c_{22}y$ with $\det \begin{pmatrix} c_{11}& c_{12}\\
c_{21}& c_{22}\end{pmatrix}=1$, the matrix $\begin{pmatrix}
a_1& a_2\\ b_1& b_2\end{pmatrix}$ becomes one of the Jordan
forms
$$\begin{pmatrix} a&0\\0& b\end{pmatrix},
\quad {\text{or}}\quad
\begin{pmatrix} a&1\\0& a\end{pmatrix}.$$
In the first Jordon case, if $a=b=0$, this is case (a). If $a\neq 0$
(or by symmetry if $b\neq 0$) we can assume that $a=1$ by a change of
basis $\{x,y,z\}\to \{\sqrt{a^{-1}}x, \sqrt{a^{-1}}y, a^{-1}z\}$.
So this is case (b). In the second Jordon case, if $a=0$, this
is case (c). If $a\neq 0$, by a change of basis, we can assume that
$a=1$, which is case (d).

Case 2: $I$ is not abelian. So we may assume that $[x,y]=y$ where
$\{x,y\}$ is a basis of $I$. Let
$z\in L\setminus I$ such that $\delta(z)=x\otimes y-y\otimes x$.
Write
$$\begin{aligned}
\; [z,x]&=a_1 x+a_2 y+a_3 z,\\
[z,y]&=b_1 x+b_2 y+b_3 z.
\end{aligned}
$$
Applying \eqref{E2.1.1} to $(a,b)=(z,x)$, we have
$$a_3(x\otimes y-y\otimes x)=\delta([z,x])=
[\delta(z), x\otimes 1+1\otimes x]=-(x\otimes y-y\otimes x).$$
Hence $a_3=-1$. A similar argument shows that $b_3=0$.  By the Jacobi
identity, we have
$$\begin{aligned}
b_1 x+b_2 y &=[z,y]=[z,[x,y]]=[[z,x],y]+[x,[z,y]]\\
&=[a_1 x+a_2 y-z,y]+[x,b_1 x+b_2 y]\\
&=a_1 y-(b_1x+b_2y)+b_2 y.
\end{aligned}
$$
Then $b_1=0$ and $a_1=b_2$. Thus we have
$$\begin{aligned}
\; [z,x]&=a x+a_2 y-z,\\
[z,y]&=ay.
\end{aligned}
$$
After replacing $z$ by $z-ax$, we have $a=0$. This is the case (e).

By the above argument, the CLAs listed are
pairwise non-isomorphic except for the isomorphism given in part (b).
\end{proof}

The enveloping algebra of the case (e) has an interesting
property that $S^2$ is not the identity, see \cite[Example 4.2]{WZZ3}.

\subsection{Dimension 4}
This is the main subsection of Section 3. We will classify all
4-dimensional anti-cocommutative CLAs. If $\delta=0$, $L$ is an
ordinary Lie algebra of dimension four and the classification is
known \cite[p. 209 Theorem 1.1]{OV} in which the base field $k$ is
${\mathbb C}$. So we assume that $\delta\neq 0$. Throughout this
subsection we assume that $L$ is a 4-dimensional anti-commutative
CLA such that $\delta\neq 0$.

\begin{lemma}
\label{zzlem3.3} Let $I=\ker \delta$. Then $I$ is a 3-dimensional
Lie subalgebra of $L$.
\end{lemma}

\begin{proof} An easy calculation shows that $I$ is a Lie subalgebra of
$L$ by using \eqref{E2.1.1}. It remains to show
that $\dim I> 2$. Let $H=U(L)$. By Lemma \ref{zzlem2.5}(g),
$p_2'(H)=\dim L-\dim I\leq {\dim I\choose 2}$. If $\dim I\leq 2$,
then $p_2'(H)\leq 1$ and $\dim L\leq 3$, a contradiction. Therefore
$\dim I>2$.
\end{proof}

\begin{lemma}
\label{zzlem3.4}
There are elements $x_1,x_2$ in $I$ such that
$\delta(L)=k(x_1\otimes x_2-x_1\otimes x_2)$.
\end{lemma}

\begin{proof} Let $\{x'_1,x'_2,x'_3\}$ be a basis of $I$. Since
$\dim I=3$, $\dim \delta(L)=1$. Pick $z\in L\setminus I$. Then we must have
$$\delta(z)=\sum_{i,j} a_{ij} (x'_i\otimes x'_j-x'_j\otimes x'_i),$$
where $A=(a_{ij})$ is a non-zero $3\times 3$ anti-symmetric matrix.
Obviously, $\delta(L)$ is spanned by $\delta(z)$. Now $A$ has
eigenvalues $0, \lambda, -\lambda$ for some $\lambda\neq 0$. By
replacing $z$ with $z/\sqrt[3]{\lambda}$, we can assume that $A$ has
eigenvalues $0, 1, -1$. By linear algebra, there exists an invertible
matrix $P$ such that
$$PAP^T=\left(   \begin{array}{ccc}
0 & 1 & 0\\
-1 &0 & 0\\
0 &0& 0
\end{array}
\right).$$
By setting $(x_1, x_2, x_3)^T=P^{-1}(x'_1, x'_2, x'_3)^T$, we get
$$\delta(z)=x_1\otimes x_2-x_1\otimes x_2.$$
This completes the proof.
\end{proof}

Here is the main result of this section. The proof is computation
and some details are easy to check.

\begin{theorem}
\label{zzthm3.5} Let $L$ be an anti-cocommutative CLA of dimension
four such that $\delta\neq 0$. Then there is a basis $\{x_1, x_2,
x_3, z\}$ such that the coalgebra structure is given by
$\delta(x_i)=0$ and $\delta(z)=x_1\otimes x_2-x_2\otimes x_1$. The
Lie algebra structure of $L$ is given by one of the following:
\begin{enumerate}
\item
$$\begin{aligned}
\; [x_2,x_1]&=x_2,\\
[x_3,x_1]&=[x_3, x_2]=0,\\
[z,x_1]&=z+ ax_1+cx_2,\\
[z, x_2]&= ax_2,\\
[z, x_3]&= bx_2.
\end{aligned}$$
where $(a, b)=(1,1), (1, 0), (0, 1)$ or $(0, 0)$ and $c\in k$.
\item
$L=B((a_{ij}))$ is determined by
$$\begin{aligned}
\; [x_2,x_1]&=[x_3,x_1]=[x_3, x_2]=0,\\
[z,x_1]&=a_{11}x_1+a_{12}x_2+a_{13}x_3,\\
[z, x_2]&= a_{21}x_1+a_{22}x_2+a_{23}x_3,\\
[z, x_3]&= a_{31}x_1+a_{32}x_2+a_{33}x_3.
\end{aligned}$$
where $(a_{ij})$ is a $3\times 3$ matrix over $k$. The CLA
$B((a_{ij}))$ is isomorphic to $B((b_{ij}))$ if and only if the
matrix $(a_{ij})$ is similar to $\lambda (b_{ij})$ for some
$\lambda\in k^\times$.

\item
$$\begin{aligned}
\;[x_2,x_1]&=0\\
[x_3,x_1]&=x_2\\
[x_3,x_2]&=0\\
[z,x_1]& =ax_1+bx_3\\
[z,x_2]& =x_2\\
[z,x_3]& =cx_1+(1-a)x_3.
\end{aligned}
$$
where $a, b, c\in k$.
\item
$$\begin{aligned}
\;[x_2,x_1]&=0\\
[x_3,x_1]&=x_2\\
[x_3,x_2]&=0\\
[z,x_1]& =ax_1+bx_3\\
[z,x_2]& =0\\
[z,x_3]& =cx_1-ax_3.
\end{aligned}
$$
where $a, b, c\in k$.
\item
$$\begin{aligned}
\; [x_2,x_1]&=[x_3, x_2]=0,\\
[x_3, x_1]&=x_1,\\
[z, x_1]&=ax_1,\\
[z, x_2]&=bx_1,\\
[z, x_3]&= -z+cx_1+ax_3
\end{aligned}$$
where $(a, b)=(1,1), (1, 0), (0, 1)$ or $(0, 0)$ and $c\in k$.
\item
$$\begin{aligned}
\; [x_2,x_1]&=0,\\
[x_3, x_1]&= x_1+x_2,\\
[x_3, x_2]&=x_2,\\
[z, x_1]&=[z, x_2]=0,\\
[z, x_3]&= -2z
\end{aligned}$$

\item
$$\begin{aligned}
\; [x_2,x_1]&=0,\\
[x_3, x_1]&= x_1,\\
[x_3, x_2]&=-x_2,\\
[z, x_1]&=ax_1+cx_2,\\
[z, x_2]&=bx_1,\\
[z, x_3]&= 0.
\end{aligned}$$
wwhere $(a, b)=(1,1), (1, 0), (0, 1)$ or $(0, 0)$ and $c\in k$.

\item
$L=H(\lambda,a)$ is determined by
$$\begin{aligned}
\; [x_2,x_1]&=0,\\
[x_3, x_1]&= x_1,\\
[x_3, x_2]&=\lambda x_2,\\
[z, x_1]&=ax_2,\\
[z, x_2]&=ax_1,\\
[z, x_3]&= (-1-\lambda)z.
\end{aligned}$$
where $\lambda\in k\setminus\{0, -1\}$ and $a\in \{0, 1\}$. The CLA
$H(\lambda,a)$ is isomorphic to $H(\lambda',a')$ if and only if
$a'=a$ and $\lambda'=\lambda$ or $\lambda'=\lambda^{-1}$.
\end{enumerate}
The CLAs listed above are pairwise non-isomorphic except for the
isomorphisms given in parts (b,h).
\end{theorem}

\begin{proof}
By the previous two lemmas, for any $4$-dimensional anti-cocommutative
CLA $L$ with non-zero $\delta$, we can choose a basis
$\{x_1, x_2, x_3, z\}$ such that $I=\ker\delta$ is spanned by
$\{x_1, x_2, x_3\}$ and $\delta(z)=x_1\otimes x_2-x_1\otimes x_2$.

Write
$$\begin{aligned}
\;[x_2,x_1]&=a_1x_1+b_1x_2+c_1x_3\\
[x_3,x_1]&=a_2x_1+b_2x_2+c_2x_3\\
[x_3,x_2]&=a_3x_1+b_3x_2+c_3x_3\\
[z,x_1]& =e_1 z+\phi_1=e_1 z+ f_1x_1+g_1x_2+h_1x_3\\
[z,x_2]& =e_2 z+\phi_2=e_2 z+ f_2x_1+g_2x_2+h_2x_3\\
[z,x_3]& =e_3 z+\phi_3=e_3 z+ f_3x_1+g_3x_2+h_3x_3.
\end{aligned}
$$
Applying \eqref{E2.1.1} to $(z,x_i)$ for $i=1,2,3$, one obtains that
$c_1=c_2=c_3=0$ and that $e_1=b_1$, $e_2=-a_1$, $e_3=-a_2-b_3$. Thus
\eqref{E2.1.1} holds if and only if the Lie bracket satisfies
$$\begin{aligned}
\;[x_2,x_1]&=a_1x_1+b_1x_2\\
[x_3,x_1]&=a_2x_1+b_2x_2\\
[x_3,x_2]&=a_3x_1+b_3x_2\\
[z,x_1]& =b_1 z+ f_1x_1+g_1x_2+h_1x_3\\
[z,x_2]& =-a_1 z+ f_2x_1+g_2x_2+h_2x_3\\
[z,x_3]& =(-a_2-b_3) z+ f_3x_1+g_3x_2+h_3x_3.
\end{aligned}
$$
In particular, $J:=kx_1+kx_2$ is a Lie subalgebra. So $J$
is either abelian or sovlable.

If $J$ is solvable, we may assume that $[x_2,x_1]=x_2$.
Since $I$ is a Lie subalgebra,
$$\begin{aligned}
a_3x_1+b_3x_2&=[x_3,x_2]=[x_3,[x_2,x_1]]\\
&=[[x_3,x_2],x_1]+[x_2,[x_3,x_1]]\\
&=[a_3x_1+b_3x_2,x_1]+[x_2,a_2x_1+b_2x_2]\\
&=b_3x_2+a_2 x_2.
\end{aligned}
$$
This implies that $a_2=a_3=0$.

If $b_2\neq 0$ or $b_3\neq 0$, after a base change, we have
$b_2=b_3=0$. The Jacobi identity for other elements implies that
$$\begin{aligned}
\;[x_2,x_1]&=x_2\\
[x_3,x_1]&=0\\
[x_3,x_2]&=0\\
[z,x_1]& =z+f_1x_1+g_1x_2+h_1x_3\\
[z,x_2]& =f_1x_2\\
[z,x_3]& =g_3x_2
\end{aligned}
$$
for $f_i,g_i,h_i\in k$. There are a few cases. First of all, we can
make $h_1=0$ by replacing $z$ with $z+h_1x_3$. If $f_1\neq 0$, we
can assume $f_1=1$ by replacing $z, x_2$ with $\frac{1}{f_1}z,
\frac{1}{f_1}x_2$, respectively. If $g_3\neq 0$, we can also
normalize it to be $1$ by replacing $x_3$ with $\frac{1}{g_3}x_3$.
All non-isomorphic Lie algebras are now listed in part (a).

If $J$ is abelian, $[x_2,x_1]=0$. If $I$ is abelian, then
$[z,x_i]=\phi_i$ defines a Lie algebra and for any $\phi_i\in I$,
$i=1,2,3$. This is part (b). Further classification can be made by
taking the Jordan form of the coefficient matrix of
$\{\phi_1,\phi_2,\phi_3\}$. This is a linear algebra classification
and, to save space, we will not list all possibilities.

For the rest of the proof we assume that $J$ is abelian and $I$ is
not abelian. Up to a change of basis $\{x_1,x_2\}$, we may assume
that
$$\begin{aligned}
\; [x_3,x_1]&=a x_1+ b x_2\\
[x_3,x_2]&= c x_2
\end{aligned}
$$
where
$$\begin{pmatrix} a& b\\0&c\end{pmatrix}
=\begin{pmatrix} 0& 1\\0&0\end{pmatrix}
{\text{  or  }}
\begin{pmatrix} 1& 0\\0&0\end{pmatrix}
{\text{  or  }}
\begin{pmatrix} 1& 1\\0&1\end{pmatrix}
{\text{  or  }}
\begin{pmatrix} 1& 0\\0&\lambda\end{pmatrix}
$$
where $\lambda \in k^\times$.

Recall that
$$\begin{aligned}
\; [z,x_1]& =\phi_1=f_1x_1+g_1x_2+h_1x_3\\
[z,x_2]& =\phi_2=f_2x_1+g_2x_2+h_2x_3\\
[z,x_3]& =(-a-c) z+ \phi_3=(-a-c)z+f_3x_1+g_3x_2+h_3x_3.
\end{aligned}
$$
The Jacobi identity implies that
$$\begin{aligned}
0&=[\phi_2,x_1]+[x_2,\phi_1]\\
(2a+c)\phi_1+b \phi_2&=[\phi_3,x_1]+[x_3,\phi_1]\\
(a+2c)\phi_2 &=[\phi_3,x_2]+[x_3,\phi_2].
\end{aligned}
$$
which completely determine the Lie algebra $L$.

If $\begin{pmatrix} a& b\\0&c\end{pmatrix} =\begin{pmatrix} 0&
1\\0&0\end{pmatrix}$, then we have
$$\begin{aligned}
\;[x_2,x_1]&=0\\
[x_3,x_1]&=x_2\\
[x_3,x_2]&=0\\
[z,x_1]& =f_1x_1+g_1x_2+h_1x_3\\
[z,x_2]& =(h_3+f_1)x_2\\
[z,x_3]& =f_3x_1+g_3x_2+h_3x_3.
\end{aligned}
$$
First, we can make $g_1=g_3=0$ by replacing $z$ with
$z+g_3x_1-g_1x_3$. If $h_3+f_1\neq 0$, we can make $h_3+f_1=1$ by
replacing $x_1, x_2, z$ with $\frac{1}{\sqrt{h_3+f_1}}x_1,
\frac{1}{\sqrt{h_3+f_1}}x_2, \frac{1}{h_3+f_1}z$. After recycling
the letters $a,b,c$, we obtain part (c). If $h_3+f_1=0$, this is
part (d) after recycling the letters $a,b,c$.

For the rest we use similar computations, so some details are
omitted.

If $\begin{pmatrix} a& b\\0&c\end{pmatrix}=
\begin{pmatrix} 1& 0\\0&0\end{pmatrix}$, then we have
$$\begin{aligned}
\;[x_2,x_1]&=0\\
[x_3,x_1]&=x_1\\
[x_3,x_2]&=0\\
[z,x_1]& =f_1x_1\\
[z,x_2]& =f_2x_1\\
[z,x_3]& =-z+f_3x_1+g_3x_2+f_1x_3.
\end{aligned}
$$
First,  $g_3$ can be made $0$ by replacing $z$ with $z-g_3x_2$. If
$f_1\neq 0$, we can assume $f_1=1$ by replacing $z, x_1$ with
$\frac{1}{f_1}z, \frac{1}{f_1}x_1$, respectively. If $f_2\neq 0$, we
can also normalize it to be $1$ by replacing $x_1, x_2$ with
$\sqrt{f_2}x_1, \frac{1}{\sqrt{f_2}}x_2$, respectively. All cases
are listed in part (e).

If $\begin{pmatrix} a& b\\0&c\end{pmatrix}=
\begin{pmatrix} 1& 1\\0&1\end{pmatrix}$, then we have
$$\begin{aligned}
\;[x_2,x_1]&=0\\
[x_3,x_1]&=x_1+x_2\\
[x_3,x_2]&=x_2\\
[z,x_1]& =f_1x_1+f_1x_2\\
[z,x_2]& =f_1x_2\\
[z,x_3]& =-2z+f_3x_1+g_3x_2+2f_1x_3.
\end{aligned}
$$
First, we can make $f_3=g_3=0$ by replacing $z$ with
$z-f_3x_1-(f_3+g_3)x_2$. Then we can assume that $f_1=0$ by
replacing $z$ with $z-f_1x_3$. This part (f).

Finally, if $\begin{pmatrix} a& b\\0&c\end{pmatrix}=
\begin{pmatrix} 1& 0\\0&\lambda\end{pmatrix}$ where $\lambda
\neq 0$, then we have
$$\begin{aligned}
\;[x_2,x_1]&=0\\
[x_3,x_1]&=x_1\\
[x_3,x_2]&=\lambda x_2\\
[z,x_1]& =f_1x_1+g_1x_2\\
[z,x_2]& =f_2x_1+g_2x_2\\
[z,x_3]& =(-1-\lambda)z+f_3x_1+g_3x_2+(1+\lambda)f_1x_3
\end{aligned}
$$
such that $(1+\lambda)g_1=(1+\lambda)f_2=0$ and $\lambda(1+\lambda)
f_1=(1+\lambda)g_2$.

Suppose $\lambda=-1$. Then we can make $g_2=f_3=g_3=0$ by replacing
$z$ with $z+f_3x_1-g_3x_2+g_2x_3$. Then, if $f_1\neq 0$, we can make
$f_1=1$ by replacing $x_2, z$ with $\frac{1}{f_1}x_2,
\frac{1}{f_1}z$. If if $f_2\neq 0$, we can make $f_2=1$ by replacing
$x_1, x_2$ with $\sqrt{f_2}x_1, \frac{1}{\sqrt{f_2}}x_2$ (Notice
that $f_1$ will not change). This is part (g).

Suppose $\lambda\neq -1$. We can also assume that $f_3=g_3=0$ by
replacing $z$ with $z-\frac{f_3}{\lambda}x_1-g_3x_2$. Now the last
three relations become
$$\begin{aligned}
\;[z,x_1]& =f_1x_1+g_1x_2\\
[z,x_2]& =g_1x_1+\lambda f_1x_2\\
[z,x_3]& =(-1-\lambda)z+(1+\lambda)f_1x_3
\end{aligned}
$$
Then by replacing $z$ with $z-f_1x_3$, we can assume that $f_1=0$.
This is part (h). We finish the proof.
\end{proof}

\section{Classification of Primitively-thin algebras of dimension at most 4}
\label{zzsec4}

Recall that a connected Hopf algebra $H$ is primitively-thin if
$p(H)=2$.

If $\GKdim H=2$, then $H=U({\mathfrak g})$ for a Lie algebra of
dimension 2, which follows from Lemmas \ref{zzlem1.3}(e) and 2.6(c).

\subsection{Primitively-thin Hopf algebras of GK-dimension 3}
\label{zzsec4.1} We recall the definition of two classes of Hopf
algebras from \cite{Zh2}, which will be used later in the
classification.

\begin{example}\label{zzex4.1}
Let $A$ be the algebra generated by elements $X, Y, Z$ satisfying
the following relations,
\begin{align*}
[X, Y]&=0,\\
[Z, X]&=\lambda_1 X+\alpha Y,\\
[Z, Y]&=\lambda_2 Y,
\end{align*}
where $\alpha=0$ if $\lambda_1\neq \lambda_2$ and $\alpha=0$ or $1$
if $\lambda_1= \lambda_2$. Then $A$ becomes a Hopf algebra via
\begin{align}
\eps(X)=0, & \quad \Delta(X)=1\otimes X+ X\otimes 1,\notag\\
\eps(Y)=0, & \quad \Delta(Y)=1\otimes Y+ Y\otimes 1,
\label{E4.1.1}\tag{E4.1.1}\\
\eps(Z)=0, & \quad \Delta(Z)=1\otimes Z+X\otimes Y-Y\otimes X+
Z\otimes 1.\notag
\end{align}
We denote this Hopf algebra by $A(\lambda_1, \lambda_2, \alpha)$. It
is easy to see that $A(\lambda_1, \lambda_2, \alpha)$ is the
enveloping algebra $U(a(\lambda_1, \lambda_2, \alpha)))$ where
$a(\lambda_1, \lambda_2, \alpha)$ is the CLA defined before Lemma
\ref{zzlem3.2}.
\end{example}

\begin{example}\label{zzex4.2}
Let $B$ be the algebra generated by elements $X, Y, Z$ satisfying
the following relations,
\begin{align*}
[X, Y]&=Y,\\
[Z, X]&=-Z+\lambda Y,\\
[Z, Y]&=0,
\end{align*}
where $\lambda \in k$. Then $B$ becomes a Hopf algebra via the
coalgebra structure given as in \eqref{E4.1.1}. We denote this Hopf
algebra by $B(\lambda)$. This algebra is the enveloping algebra
$U(b(\lambda))$ where $b(\lambda)$ is the CLA defined in Lemma
\ref{zzlem3.2}(e).
\end{example}

The following proposition is from \cite[Theorem 7.8]{Zh2}. It also follows
now from Lemma \ref{zzlem3.2} and Theorem \ref{zzthm2.7}.

\begin{proposition}
\label{zzpro4.3}
Let $H$ be a connected Hopf algebra. Then $H$ is primitively-thin of
GK-dimension $3$ if and only if $H$ is isomorphic to one of the
following:
\begin{enumerate}
\item
The Hopf algebras $A(0, 0, 0)$, $A(0, 0, 1)$, $A(1, 1, 1)$ or $A(1,
\lambda, 0)$ from Example \ref{zzex4.1} for some $\lambda\in k$;
\item The Hopf algebras $B(\lambda)$ from Example
\ref{zzex4.2} for some $\lambda\in k$.
\end{enumerate}
\end{proposition}

\subsection{Examples of connected Hopf algebras of GK-dimension 4}
\label{zzsec4.2}
We introduce four classes of connected Hopf
algebras of GK-dimension $4$. Later we will show that these classes
give a complete description of primitively-thin Hopf algebras of
GK-dimension $4$.

\begin{example}\label{zzex4.4}
Let $D$ be the algebra generated by $X, Y, Z, W$ satisfying the
following relations,
\begin{align*}
[Y, X]&=[Z, X]=[Z, Y]=0,\\
[W, X]&=a_{11}X+a_{12}Y,\\
[W, Y]&=a_{21}X+a_{22}Y,\\
[W, Z]&=(a_{11}+a_{22})Z+\xi_1X+\xi_2Y,
\end{align*}
where $a_{ij}, \xi_i\in k$. Then $D$ becomes a bialgebra via
\begin{align*}
\eps(X)=0,&\quad \Delta(X)=1\otimes X+ X\otimes 1,\\
\eps(Y)=0,&\quad \Delta(Y)=1\otimes Y+ Y\otimes 1,\\
\eps(Z)=0,&\quad \Delta(Z)=1\otimes Z+X\otimes Y-Y\otimes X+ Z\otimes 1,\\
\eps(W)=0, & \\
\Delta(W)=1&\otimes W+W\otimes 1\\
&+\theta_1(Z\otimes X-X\otimes Z+X\otimes XY+XY\otimes
X)\\
&+\theta_2(Y\otimes Z-Z\otimes Y+XY\otimes Y+Y\otimes XY),
\end{align*}
where $\theta_i\in k$ and at least one of them is non-zero. This
bialgebra is also denoted by $D(\{\theta_i\}, \{a_{ij}\},\{\xi_i\})$
if we want to indicate the parameters.
It is easy to see that the coalgebra structure is connected. Hence
the bialgebra $D$ is automatically a Hopf algebra. Note that $P(D)=
kX+kY$ and $P_2(D)=kX+kY+kZ$. Let $f$ be a Hopf algebra isomorphism
between two such Hopf algebras (or a Hopf algebra automorphism).
Then $f$ preserves the subspaces $kX+kY$ and $kX+kY+kZ$, and it is
now not hard to see that $f$ sends
\begin{align}
X&\longrightarrow c_{11}X+c_{12}Y,\notag\\
Y&\longrightarrow c_{21}X+c_{22}Y,\label{E4.4.1}\tag{E4.4.1}\\
Z&\longrightarrow c_{31}X+c_{32}Y+c_{33}Z,\notag\\
W&\longrightarrow c_{44}W+w(X,Y,Z)\notag
\end{align}
 where $c_{ij}\in k$ and $w(X,Y,Z)$ is a certain
polynomial of $X,Y,Z$. (Note that $f$ being a Hopf algebra isomorphism
implies that $c_{ij}$ and $w$ satisfy some conditions which we will
not give details here).

Using an isomorphism $f$, or equivalently, choosing a new basis
$\{X,Y,Z,W\}$ properly, one can first normalize the matrix
$(a_{ij})_{2\times 2}$ so that it becomes one of the following five:
\begin{equation}
\label{E4.4.2}\tag{E4.4.2}
\begin{pmatrix} 0& 0\\0&0\end{pmatrix}
{\text{ , }}
\begin{pmatrix} 0& 1\\0&0\end{pmatrix}
{\text{ , }}
\begin{pmatrix} 1& 0\\0&0\end{pmatrix}
{\text{  ,  }}
\begin{pmatrix} 1& 1\\0&1\end{pmatrix}
{\text{  ,  }}
\begin{pmatrix} 1& 0\\0&\lambda\end{pmatrix}
\end{equation}
where $\lambda\neq 0$ and in the last class $\begin{pmatrix} 1& 0\\0&\lambda\end{pmatrix}$
is equivalent to $\begin{pmatrix} 1& 0\\0&\lambda^{-1}\end{pmatrix}$.
The Hopf algebras are pair-wise non-isomorphic if these are in different
classes. Within any class, two Hopf algebras
$D(\{\theta_i\},\{a_{ij}\},\{\xi_i\})$s can be isomorphic for different
parameters $\{\theta_i\}$ and $\{\xi_i\}$, which is determined by the
base changes that fixes the matrix given in \eqref{E4.4.2} (or change
$\begin{pmatrix} 1& 0\\0&\lambda\end{pmatrix}$ to
$\begin{pmatrix} 1& 0\\0&\lambda^{-1}\end{pmatrix}$ if in the last class).
For example, by replacing $\{X,Y,Z,W\}$ by $\{aX, aY, a^2Z, W\}$, the
parameters $\{\theta_1,\theta_2\}$ becomes $\{a^{-3}\theta_1,
a^{-3}\theta_2\}$. This means that we may assume that $\{\theta_1,\theta_2\}
\in {\mathbb P}^1$. Dependent on the form of the matrix $(a_{ij})$
listed in \eqref{E4.4.2}, one can further decide the parameters
$\{\xi_1,\xi_2\}$ such that the Hopf algebras $D(\{\theta_i\},
\{a_{ij}\},\{\xi_i\})$ are non-isomorphic. In summary,
the isomorphism classes of $D(\{\theta_i\}, \{a_{ij}\},\{\xi_i\})$
can be completely determined by easy linear algebra.

There is another way of classifying all isomorphism classes of
$D(\{\theta_i\}, \{a_{ij}\},\{\xi_i\})$. First, by choosing the
basis $\{X,Y,Z,W\}$ properly, one can assume that $\theta_1=0$ and
$\theta_2=1$. So we can fix $\{\theta_1,\theta_2\}=\{0,1\}$. In
particular,
$$\Delta(W)=1\otimes W+W\otimes 1+Y\otimes Z-Z\otimes Y+XY\otimes Y+Y\otimes
XY.$$
Under this restriction, one can further classify the
parameters $\{(a_{ij}), \{\xi_i\}\}$. Unfortunately, then we can not
assume that the matrix $(a_{ij})$ is of one of the form given in
\eqref{E4.4.2}.

From the algebraic relations, $D(\{\theta_i\},
\{a_{ij}\},\{\xi_i\})$ is isomorphic to a universal enveloping
algebra of a Lie algebra.
\end{example}

\begin{example}\label{zzex4.5}
 Let $E$ be the algebra generated by $X,
Y, Z, W$ satisfying the following relations,
\begin{align*}
[Y, X]&=[Z, Y]=0,\\
[Z, X]&=X,\\
[W, X]&=a X,\\
[W, Y]&=b X,\\
[W, Z]&=aZ-W+\xi X+\xi' Y.
\end{align*}
where $\xi, \xi'\in k$. Then $E$ becomes a bialgebra (and then a
Hopf algebra) via
\begin{align*}
\eps(X)=0,&\quad \Delta(X)=1\otimes X+ X\otimes 1,\\
\eps(Y)=0,&\quad \Delta(Y)=1\otimes Y+ Y\otimes 1,\\
\eps(Z)=0,&\quad \Delta(Z)=1\otimes Z+X\otimes Y-Y\otimes X+ Z\otimes 1,\\
\eps(W)=0,& \\
\Delta(W)=1&\otimes W+W\otimes 1\\
&\quad +Z\otimes X-X\otimes Z+X\otimes XY+XY\otimes X.
\end{align*}
Up to a base change (by setting $W_{new}=W-\xi' Y$), we may assume
that $\xi'=0$. We denote this Hopf algebra by $E(a,b,\xi)$ where
$a,b,\xi\in k$. Using a base change
$$\begin{aligned}
X&\to c X,\\
Y&\to c^{-1}Y,\\
Z&\to Z,\\
W&\to c W,
\end{aligned}
$$
for some $c\in k$ we can re-scale $(a,b,\xi)$. The complete set of non-isomorphic
classes of $E(a,b,\xi)$ is corresponding to the following cases
$$(a,b,\xi)=\begin{cases} 
(0,0,\xi) & {\text{if $a=b=0$}},\\
(0,1,\xi)& {\text{if $a=0$ and $b\neq 0$}},\\
(1,b,\xi)& {\text{if $a\neq 0$}}.\end{cases}
$$
To unify the presentation of $\Delta(W)$, we make a change of basis
$$\begin{aligned}
X& \to Y,\\
Y& \to  X,\\
Z& \to -Z,\\
W&\to W.
\end{aligned}
$$
Under the new basis, the coalgebra structure is same except for
$\Delta(W)$, which becomes
$$\Delta(W)=1\otimes W+W\otimes 1+Y\otimes Z-Z\otimes Y+XY\otimes Y+Y\otimes
XY.$$ The algebraic relations change accordingly, which can be
easily done.

From the algebraic relations, $E(a,b,\xi)$ is isomorphic to a
universal enveloping algebra of a Lie algebra.
\end{example}

\begin{example}\label{zzex4.6}
Let $F$ be the algebra generated by $X, Y, Z, W$ satisfying the
following relations,
\begin{align*}
[Y, X]&=[Z, Y]=0,\\
[Z, X]&=Y,\\
[W, X]&=\beta Y,\\
[W, Y]&=\gamma Y,\\
[W, Z]&=\gamma Z-\frac{2}{3}Y^3+\xi X+\xi' Y.
\end{align*}
where $\beta, \gamma, \xi, \xi'\in k$. Then $F$ becomes a bialgebra
(and then a Hopf algebra) via
\begin{align}
\eps(X)=0, &\quad \Delta(X)=1\otimes X+ X\otimes 1,\notag\\
\eps(Y)=0, &\quad \Delta(Y)=1\otimes Y+ Y\otimes 1,\notag\\
\eps(Z)=0, &\quad \Delta(Z)=1\otimes Z+X\otimes Y-Y\otimes X+ Z\otimes 1,
\label{E4.6.1}\tag{E4.6.1}\\
\eps(W)=0, &\quad \notag\\
\Delta(W)=1& \otimes W+W\otimes 1\notag\\
&\quad +Y\otimes Z-Z\otimes Y+XY\otimes Y+Y\otimes XY.\notag
\end{align}
If $W$ is replaced by $W_{new}:=W+\xi' X$, then we can assume
$\xi'=0$. We denote the Hopf algebra by $F(\beta, \gamma,\xi)$. One
can make further reduction by easy linear algebra. For example, if
$\gamma\neq 0$, by replacing $X$ by $X_{new}:=X-\gamma^{-1} \beta
Y$, we have $\beta=0$. By re-scalaring, we can further assume
$\gamma=1$. If $\gamma=0$, then, by re-scalaring, we might assume
$\beta=1$. In summary, $\{\beta,\gamma\}$ is either $\{0,1\}$ or
$\{1,0\}$. This completely determines the isomorphism classes of the
Hopf algebras $F(\beta, \gamma,\xi)$.

Let $W'=W-\frac{2}{3}XY^2$. Then the algebraic relations of $F$
becomes
\begin{align*}
[Y, X]&=[Z, Y]=0,\\
[Z, X]&=Y,\\
[W', X]&=\beta Y,\\
[W', Y]&=\gamma Y,\\
[W', Z]&=\gamma Z+\xi X.
\end{align*}
Therefore the subspace generated by $\{X,Y,Z,W'\}$ is a
4-dimensional Lie algebra, say $\mathfrak g$, and $F$ is isomorphic
to the enveloping algebra $U({\mathfrak g})$ as algebras.
\end{example}

\begin{example}\label{zzex4.7}
Let $K$ be the algebra generated by $X, Y, Z, W$ satisfying the
following relations,
\begin{align*}
[Y, X]&=[Z, Y]=0,\\
[Z, X]&=X,\\
[W, X]&=-Z,\\
[W, Y]&=0,\\
[W, Z]&=W-XY^2.
\end{align*}
The coalgebra structure of $K$ is given as in \eqref{E4.6.1}. Then
$K$ becomes a Hopf algebra. Let $W'=W-\frac{1}{2} XY^2$. Then the
algebraic relations become
\begin{align*}
[Y, X]&=[Z, Y]=0,\\
[Z, X]&=X,\\
[W', X]&=-Z,\\
[W', Y]&=0,\\
[W', Z]&=W'.
\end{align*}
Therefore the subspace generated by $\{X,Y,Z,W'\}$ is a
4-dimensional Lie algebra, say $\mathfrak g$, and $K$ is isomorphic
to the enveloping algebra $U({\mathfrak g})$ as algebras.
\end{example}

\begin{proposition}
\label{zzpro4.8} The algebras $H$ defined in Examples
\ref{zzex4.4}-\ref{zzex4.7} have the following properties.
\begin{enumerate}
\item
$H$ is an iterated Ore extension $k[X][Y;\delta_1][Z;\delta_2]
[W;\sigma_3,\delta_3]$.
\item
$H$ is an Auslander regular Cohen-Macaulay domain.
\item
The global dimension and GK-dimension of $H$ is 4.
\item
$H$ is a Hopf algebra and connected as a coalgebra.
\item
$P(H)=kX+kY$. The subalgebra generated by $X,Y$ is the enveloping
algebra $U(\mathfrak g)$ where ${\mathfrak g}$ is the Lie algebra
$P(H)$.
\item
$P_2(H)=kX+kY+kZ$. The subalgebra generated by $X,Y,Z$ is the enveloping
algebra $U(L)$ where $L$ is the CLA $P_2(H)$.
\item
$H$ is not isomorphic to an enveloping algebra of either a Lie algebra
or a CLA.
\item
The lantarn ${\mathfrak L}(H)$ of $H$ is isomorphic to the graded
Lie algebra of dimension four, with a basis $\{x^*,y^*,z^*,w^*\}$,
such that $z^*=[x^*,y^*]$ and $w^*=[z^*,y^*]$, subject to the
relation $[z^*,x^*]=0= [w^*,x^*]=[w^*,y^*]$.
\end{enumerate}
\end{proposition}

\begin{proof}[Sketch of the proof]
(a) Follows from the definitions.

(b,c) These facts are true for any iterated Ore extension.

(d,e,f) These dependents on straightforward, but not trivial,
computation.

(g) This follows from parts (e,f).

(h) This was proved in Lemma \ref{zzlem1.4}(c).

We can also check it directly here. Noting that
$\{x^*,y^*,z^*,w^*\}$ can be viewed as a dual basis of $\{X,Y,Z,W\}$
in $\gr H$ and that coproducts of $X,Y,Z,W$ given in Examples
\ref{zzex4.4}-\ref{zzex4.7} match up with the Lie structure of the
${\mathfrak L}(H)$ given in part (h).
\end{proof}

\subsection{Primitively-thin Hopf algebra of GK-dimension 4, Part I}
\label{zzsec4.3} In this and the next two subsections we classify
all primitively-thin Hopf algebra of GK-dimension 4.

Let $C$ be a primitively-thin Hopf algebra of GK-dimension 3. By
Proposition \ref{zzpro4.3}, $C$ is of type $A$ or $B$ as in Examples
\ref{zzex4.1} and \ref{zzex4.2}. Let $D$ be the Hopf algebra $A(0,
0, 0)$. Then it is easy to see that $D$ is a coradically graded Hopf
algebra by setting $\deg X=\deg Y=1$ and $\deg Z=2$. Since $\gr
C\cong D$, $C$ is a so-called PBW deformation of $D$.

Let $\{x, y, z\}$ be any set of generators of $C$ such that $x, y$
are primitive and $\Delta(z)=1\otimes z+x\otimes y-y\otimes
x+z\otimes 1$. Then $C$ has a basis consists of monomials of the
form
$$x^{i_1}y^{i_2}z^{i_3}.$$
Notice that $C^+$ is spanned by $x^{i_1}y^{i_2}z^{i_3}$ with at
least one $i_k$ non-zero. Let $\overline{x}, \overline{y},
\overline{z}$ be the homogeneous elements in $D=\gr C$ corresponding
to $x, y, z$, respectively. (In fact, an easy calculation shows that
$\overline{x}, \overline{y}, \overline{z}$ can be identified with
the canonical generators $X, Y, Z$ as in the definition of $A(0, 0,
0)$). Then $\gr C$ has a basis $\{\overline{x}^{i_1}
\overline{y}^{i_2} \overline{z}^{i_3}\}$. Now we have a $k$-space
isomorphism from $C$ to $\gr C$ by sending $x^{i_1}y^{i_2}z^{i_3}$
to $\overline{x}^{i_1} \overline{y}^{i_2} \overline{z}^{i_3}$.
Clearly this isomorphism maps $C^+$ onto $D^+$.  From now on we
identify $C$ with $D=\gr C$ as $k$-spaces by this isomorphism, and
we will abuse the notation by dropping the bars for the generators
$\overline{x}, \overline{y}, \overline{z}$ of $\gr C$. Define $\deg
x^{i_1}y^{i_2}z^{i_3}=i_1+i_2+2i_3$. This grading agrees with the
natural grading on $D=\gr C$. Moreover, by the defining relations of
$C$, it is easy to check that
$$\Delta_C(a)=\Delta_D(a)+ldt,$$
where $a\in C$ and $ldt$ denotes terms with degrees lower than $\deg
a$. As a consequence, we can think about $\Omega C$ and $\Omega D$,
the cobar constructions of $C$ and $D$, as the same graded
$k$-spaces with two differentials $\partial_C$ and $\partial_D$.
Moreover, $\partial_D$ respects the grading and
$$\partial^n_C(b)=\partial^n_D(b)+ldt,$$
where $b\in (C^+)^{\otimes n}$.

\begin{lemma}\label{zzlem4.9}
Let $D=A(0,0,0)$. Then $\dim_k \h^2(\Omega D)=2$ and $\h^2(\Omega
D)$ is spanned by the classes of 2-cocycles $(u)$ and $(t)$, where
\begin{equation}
\label{E4.9.1}\tag{E4.9.1} u=z\otimes x-x\otimes z+xy\otimes
x+x\otimes xy,
\end{equation}
\begin{equation}\label{E4.9.2}\tag{E4.9.2}
t=y\otimes z-z\otimes y+xy\otimes y+ y\otimes xy.
\end{equation}
\end{lemma}
\begin{proof}
As mentioned above, $D$ is a graded Hopf algebra. Let $A$ be the
graded dual of $D$. Since $D$ is commutative, $A$ is cocommutative.
By Proposition \ref{zzpro1.1}, $A=U({\mathfrak L})$ for some graded
Lie algebra ${\mathfrak L}$. Since $D$ is coradically graded as a
coalgebra, $A$ is generated in degree one. This implies that
${\mathfrak L}$ is the 3-dimensional Heisenberg Lie algebra. Or
equivalently, $A$ is generated by two degree one elements $x_1$ and
$x_2$ with relations
$$x_1^2x_2+x_2x_1^2=2x_1x_2x_1, \,\,\,\,x_2^2x_1+x_1x_2^2=2x_2x_1x_2.$$
By \cite[Lemma 8.6 (c)]{LPWZ2}, $B^\#A\cong \Omega C$ as DG
algebras, where $B^\#A$ is the graded dual of the bar construction
of $A$. On the other hand, by \cite[Lemma 4.2]{LPWZ},
$\h^\bullet(B^\#A)\cong \Ext^\bullet_A(k_A, k_A)$. As a consequence,
$$\dim_k \h^2(\Omega D)=\dim_k \Ext^2_A(k_A, k_A)=2.$$

We introduce a ${\mathbb Z}^2$-grading on $D$ by setting $\deg_2
x=(1, 0)$, $\deg_2 y=(0, 1)$ and $\deg_2z=(1, 1)$. Then it is clear
that $D$ is a ${\mathbb Z}^2$-graded Hopf algebra and therefore the
differentials of $\Omega D$ preserves the ${\mathbb Z}^2$-grading. A
direct computation shows that both $u$ and $t$ are 2-cocycles. Now
if $u$ is 2-coboundary, it must be a linear combination of
$\partial^1(x^2y)=\delta(x^2y)=x^2\otimes y+2xy\otimes x+2x\otimes
xy+y\otimes x^2$ and $\partial^1(xz)=\delta(xz)=x\otimes z+z\otimes
x+x^2\otimes y+x\otimes xy$. An easy calculation shows this is
impossible. Hence the class $(u)$ is a non-zero element in
$\h^2(\Omega D)$. Similarly, one can show that the class $(t)$ is
also non-zero in $\h^2(\Omega D)$. Moreover, they are linearly
independent since they have different ${\mathbb Z}^2$-degrees. This
completes the proof.
\end{proof}

Recall that monomials of the form $x^i y^j z^k$ in $C$ are
identified with monomials $x^i y^j z^k$ in $D$.

\begin{proposition}
\label{zzpro4.10} Retain the above notation. Then $\h^2(\Omega C)$
is spanned by the classes of 2-cocycles $(u)$ and $(t)$, where $u$
and $t$ are given as in \eqref{E4.9.1}-\eqref{E4.9.2}, and $\dim_k
\h^2(\Omega C)=2$.
\end{proposition}

\begin{proof}
Let $w$ be a non-zero linear combination of $u$ and $t$. Then $w$ is
homogeneous of degree $3$ and $\partial_C^2(w)=\partial_D^2(w)=0$ by
direct calculation (which only uses \eqref{E4.1.1}). Suppose that
there exists $c\in C^+$ such that $w=\partial^1_{C}(c)$. Also, we
can write $\partial^1_{C}(c)=\partial^1_{D}(c)+v$, where $v\in
(C^+)^{\otimes 2}$ has degree less than the degree of $c$. If the
degree of $c$ is larger than $3$, then $\deg \partial^1_{D}(c)=\deg
c>3$ as $\partial^1_{D}$ is homogeneous and $\partial^1_{D}(c)\neq
0$. Then
$$\deg w=\deg \partial^1_{C}(c)=\deg \partial^1_{D}(c)>3,$$
a contradiction. Therefore $\deg c\leq 3$ and consequently, $v$ has
degree less than $3$.

Since $\partial_D^2\partial_C^1(c)=\partial_D^2(w)=0$, we have
$\partial_D^2(\partial^1_{D}(c)+v)=0$. Hence $\partial_D^2(v)=0$.
But $\deg v< 3$, so by Lemma \ref{zzlem4.9}, there exists $c'\in
C^+$ such that $v=\partial^1_D(c')$. As a consequence,
$w=\partial^1_{D}(c)+v=\partial^1_D(c+c')$, which is a
contradiction.

Now, we have shown that $\h^2(\Omega C)$ is at least of dimension
two and $(u), (t)$ are linearly independent in $\h^2(\Omega C)$. On
the other hand, by a standard spectral sequence argument \cite[Theorem 5.5.1]{We}, we have $\dim_k\h^2(\Omega
C)\le\dim_k\h^2(\Omega D)=2$. This
completes the proof.
\end{proof}

\begin{lemma}\label{zzlem4.11}
Let $H$ be a connected coalgebra and $K$ a proper subcoalgebra of
$H$. Let $N$ be the smallest number such that $K_N\subsetneq H_N$
and suppose that $N\ge2$, then $\delta$ induces a injective
$k$-linear map from $H^+_N/K^+_N$ to $\h^2(\Omega K)$.
\end{lemma}

\begin{proof}
By the choice of $N$, we see that, for any $g\in H^+_N$,
$$\delta(g)=\Delta(g)-(1\otimes g+g\otimes 1)\in H^+_{N-1}\otimes H^+_{N-1}=
K^+_{N-1}\otimes K^+_{N-1}.$$ Hence
$\partial^2_K(\delta(g))=\partial^2_K\partial_H^1(g)=
\partial^2_H\partial_H^1(g)=0$,
which means that $\delta(g)$ is a 2-cocycle in the complex $\Omega
K$. Hence $\delta$ defines a $k$-linear map from $H^+_N\to
\h^2(\Omega K)$. For any element $g\in K^+_{N}$,
$\delta(g)=\partial_K^1(g)$ is a 2-boundary in the complex $\Omega
K$, whence it is zero in $\h^2(\Omega K)$. Thus $\delta$ induces a
$k$-linear map from $H^+_N/K^+_N\to \h^2(\Omega K)$.

If $g\in H^+_N\setminus K^+_N$, we claim that $\delta(g)$ represents
a non-zero cohomology class in $\h^2(\Omega K)$. If not, there is
$w\in K^+$ such that $\partial_K^1(w)=\delta(w)=\delta(g)$. As a
consequence, $\Delta(g-w)=1\otimes (g-w)+(g-w)\otimes 1$, i.e. $g-w$
is a primitive element in $H$. By the fact that $H^+_1=K^+_1$,
$g-w\in K^+_1$. But this would imply that $g\in K^+$, which
contradicts the choice of $g$. Therefore the map from $H^+_N/K^+_N$
to $\h^2(\Omega K)$ is injective.
\end{proof}

\begin{theorem}
\label{zzthm4.12}
Suppose that $H$ is a primitively-thin Hopf algebra of GK-dimension
$4$. Then for any linearly independent primitive elements $x ,y$,
there exists $z\in H$ such that
\begin{equation}\label{E4.12.1}\tag{E4.12.1}
\Delta(z)=1\otimes z+x\otimes y-y\otimes x+z\otimes 1.
\end{equation}
For any such $z$, the algebra $C$ generated by $\{x, y, z\}$ is a
Hopf subalgebra of GK-dimension $3$. Moreover, there exists $w\in H$
such that
\begin{equation}\label{E4.12.2}\tag{E4.12.2}\Delta(w)=1\otimes
w+\theta_1u+\theta_2t+w\otimes 1,\end{equation} where $u$ and $t$
are given as in \eqref{E4.9.1}-\eqref{E4.9.2} and one of the scalars
$\theta_1, \theta_2$ is non-zero. For any such $w$, the set $\{x, y,
z, w\}$ generates $H$.
\end{theorem}

\begin{proof}
By \cite[Proposition 7.10]{Zh2}, we can find a Hopf subalgebra $C$
of GK-dimension $3$. By Proposition \ref{zzpro4.3}, there is an
$z\in C$ such that $\Delta(z)$ is of the form \eqref{E4.12.1}. Since
$P(C)=kx\oplus ky=P(H)$, $C_1=H_1$.

Let $N\geq 2$ be the smallest integer such that $C_N\subsetneq H_N$.
By \cite[Lemma 5.3.2]{Mo}, there exists $w'\in H_N\setminus C_N$
such that $\Delta(w')=1\otimes w'+w'\otimes 1+f$ where $f\in
C_{N-1}\otimes C_{N-1}$. Without loss of generality, we assume that
$w'\in H^{+}$.

By Lemma \ref{zzlem4.11}, $f$ represents a non-zero cohomology class
in $\h^2(\Omega C)$. By Proposition \ref{zzpro4.10}, the cohomology
classes in $\h^2(\Omega C)$ represented by $f$ is a non-zero linear
combination of $(u)$ and $(t)$. Hence there exists $v\in C^+$ and
$\theta_1, \theta_2\in k$ such that
$f=\partial^1(v)+\theta_1u+\theta_2t$, where at least one of
$\theta_i$ is non-zero. Let $w=w'+v$. Then $w\notin C$ and
$\Delta(w)=1\otimes w+\theta_1u+\theta_2t+w\otimes 1$.

Next we have to show that $H$ is generated by $x$, $y$, $z$ and $w$.
Let $K$ be the subalgebra of $H$ generated by $x$, $y$, $z$ and $w$.
Then it is easy to check that $K$ is a sub-bialgebra and thus a Hopf
subalgebra of $H$. By the construction of $K$, $C\subsetneq K$. By
\cite[Lemma 6.8]{Zh2}, $\GKdim \gr K\ge \GKdim \gr C+1= 4$. On the
other hand, $\GKdim\gr K=\GKdim K\le \GKdim H=4$ since $K\subset H$.
Hence $\GKdim K=4$. Now it follows from \cite[Lemma 7.4]{Zh2} that
$K=H$. This completes the proof.
\end{proof}

As a direct consequence, we have the following proposition.

\begin{corollary}\label{zzcor4.13}
Let $H$ be a commutative, connected, primitively-thin Hopf algebra
of GK-dimension $4$. Then $H$ is isomorphic to $D(\{0,1\},
\{0\},\{0\})$.
\end{corollary}

\begin{proof}
By Theorem \ref{zzthm4.12}, there is a surjective Hopf map from
$D(\{\theta_i\}, \{0\},\{0\})$ to $H$ sending $X, Y, Z, W$ to $x, y,
z, w$, respectively, for some $\{\theta_1,\theta_2\}$. The map must
be an isomorphism since both $D$ and $H$ are domains of GK-dimension
$4$. By definition, $D(\{\theta_i\}, \{0\},\{0\})$ is a graded Hopf
algebra with $\deg X=\deg Y=1$, $\deg Z=2$ and $\deg W=3$. Hence the
graded dual $H^*$ is a graded commutative Hopf algebra, which must
be isomorphic to the enveloping algebra $U({\mathfrak L})$ for some
graded Lie algebra generated by two elements in degree 1. Such a Lie
algebra is unique (up to isomorphism) and is given in Proposition
\ref{zzpro4.8}(h). Therefore $H$ is isomorphic to $U({\mathfrak
L})^*$, which is isomorphic to $D(\{0,1\}, \{0\},\{0\})$.
\end{proof}

\begin{lemma}
\label{zzlem4.14}
Retain the notation in Theorem \ref{zzthm4.12} for parts (b,c).
\begin{enumerate}
\item
The Hopf algebra $D:=D(\{0,1\}, \{0\},\{0\})$ is coradically graded
by setting $\deg X=\deg Y=1$, $\deg Z=2$ and $\deg W=3$.
\item
For any connected Hopf algebra $H$ of GK-dimension $4$ with
$\dim_kP(H)=2$, $\gr H$ is isomorphic to $D(\{0,1\}, \{0\},\{0\})$.
\item
Working with $\gr C$ (isomorphic to $A(0, 0, 0)$) and $\gr H$, we
have $C_2=H_2$ and $H^+_3/C_3^+$ is one-dimensional, which is
spanned by the image of $w$.
\end{enumerate}
\end{lemma}

\begin{proof} Let $D$ denote the Hopf algebra $D(\{0,1\},
\{0\},\{0\})$.

(a) One can check directly that $\gr D=D$. Hence $D$ is coradically
graded.

(b) By \cite[Theorem 1.2]{Zh2}, $\gr H$ is commutative, and it is
still connected and primitively-thin. The assertion follows from
Corollary \ref{zzcor4.13}.

(c) We may replace $\gr C$ by $A(0,0,0)$ and $\gr H$ by $D$. Then
the assertion follows by an easy computation.
\end{proof}

\begin{proposition}\label{zzpro4.15}
Retain the notation in Theorem $\ref{zzthm4.12}$. Then $H$ has a
$k$-basis of the form
$$\{x^{i_1}y^{i_2}z^{i_3}w^{i_4}\mid i_1, i_2, i_3, i_4\geq 0\}.$$
\end{proposition}

\begin{proof} Let $X, Y, Z, W$ be the elements in $\gr H$
corresponding to elements $x, y, z, w$ in $H$. Then $\{X, Y, Z, W\}$
generates $\gr H$ as an algebra by Theorem $\ref{zzthm4.12}$. By
Lemma \ref{zzlem4.14}(b), $\gr H\cong D$, so $X, Y, Z, W$ satisfy
the defining relations of $D(\{0,1\}, \{0\},\{0\})$ given in Example
\ref{zzex4.4}. As a consequence, $\gr H$ has a $k$-basis of the form
$$\{X^{i_1}Y^{i_2}Z^{i_3}W^{i_4}\mid i_1, i_2, i_3, i_4\geq 0\}.$$
Now the result follows.
\end{proof}

In Theorem \ref{zzthm4.12}, the Hopf subalgebra $C$ is
primitively-thin. Hence by Proposition \ref{zzpro4.3}, $C$ must be
isomorphic to either $A(\lambda_1, \lambda_2, \alpha)$ or
$B(\lambda)$.

\subsection{Primitively-thin Hopf algebra of GK-dimension 4, Part II}
\label{zzsec4.4} In this subsection we show that $B(\lambda)$ can
not appear as a Hopf subalgebra of a primitively-thin Hopf algebra
of GK-dimension 4. We start with an easy observation.

\begin{lemma}
\label{zzlem4.16} Let $x$ and $y$ be primitive elements. Then
\begin{equation}
\label{E4.16.1}\tag{E4.16.1} \delta(xy^2)=y^2\otimes x+x\otimes y^2
+2(xy\otimes y+y\otimes xy),
\end{equation}
\begin{equation}
\label{E4.16.2}\tag{E4.16.2} \delta(x^2y)=y\otimes x^2+x^2\otimes y
+2(xy\otimes x+x\otimes xy),
\end{equation}
\begin{equation}
\label{E4.16.3}\tag{E4.16.3} \delta(y^3)=3(y\otimes y^2+y^2\otimes
y).
\end{equation}
\end{lemma}

\begin{proposition}\label{zzpro4.17}
Retain the notation in Theorem \ref{zzthm4.12}. Then the Hopf
subalgebra $C$ can not be isomorphic to $B(\lambda)$.
\end{proposition}

\begin{proof}
Suppose to the contrary that $C$ is isomorphic to $B(\lambda)$ for
some $\lambda\in k$. Therefore we can assume that $x, y, z\in C$
satisfies the relations listed in Example \ref{zzex4.2}.

First we assume that $\theta_1$ is not zero. By dividing $w$ with
$\theta_1$ we may assume that $\Delta(w)=1\otimes
w+u+\theta_2t+w\otimes 1$. Using \eqref{E4.12.2}, \eqref{E4.16.1},
\eqref{E4.16.3} and the reltaions of $B(\lambda)$, we have
\begin{align*}
\delta([w, y])&=\Delta([w,y])-[w,y]\otimes 1-1\otimes [w,y]\\
&=[\Delta(w),\Delta(y)]-[w,y]\otimes 1-1\otimes [w,y]\\
&=z\otimes [x, y]-[x, y]\otimes z\\
&\quad +[xy, y]\otimes x+ x \otimes [xy, y]+xy\otimes [x, y]+[x, y]\otimes xy\\
&\quad +\theta_2([xy, y]\otimes y+y\otimes[xy, y])\\
&=-t+\delta(xy^2)+\frac{\theta_2}{3}\delta(y^3).
\end{align*}
Let $w'=xy^2+\frac{\theta_2}{3}y^3-[w, y]$, then $\delta(w')=t$,
and, whence,
$$\Delta(w')=1\otimes w'+t+w'\otimes 1.$$
Now, under the map given in Lemma \ref{zzlem4.11}, the elements $w$
and $w'$ are mapped to $(u)+\theta_2(t)$ and $(t)$, respectively.
Hence $\dim_k H^+_3/C_3^+=2$, which contradicts Lemma
\ref{zzlem4.14}(c).

Next we assume that $\theta_1=0$ and $\theta_2\neq0$. By dividing
$w$ with $\theta_2$ we may assume that $\Delta(w)=1\otimes
w+t+w\otimes 1$. A similar calculation shows that $[w, x]+2w\in
P(H)$ and $[w, y]-\frac{1}{3}y^3\in P(H)$. As a consequence,
\begin{align}\label{E4.17.1}\tag{E4.17.1}
\delta([w, z])&=[w, x]\otimes y-y\otimes [w, x]+ x\otimes [w, y]-[w, y]\otimes x\\
\nonumber &\quad +[y, z]\otimes z-z\otimes [y, z]+[xy, z]\otimes y+y\otimes [xy, z]\\
\nonumber &\quad +[t, x\otimes y-y\otimes x]\\
 \nonumber&=-2w\otimes y+2y\otimes w+\sum_s\alpha_sf_s\otimes g_s
 \end{align}
where $\alpha_s\in k$, $f_s, g_s$ are monomials of the form
$x^{i_1}y^{i_2}z^{i_3}$. Let $(\gr H)^n$ denote the degree $n$ piece
of the graded Hopf algebra $\gr H$. Since $\gr H$ is commutative,
$[w, z]$ represents an element $V\in \gr H(4)$. Let $X\in \gr
H(1)$, $Y\in \gr H(1)$, $Z\in \gr H(2)$ and $W\in \gr H(3)$ be
the homogeneous elements in $\gr H$ corresponding to elements $x, y,
z$ and $w$ in $H$, respectively. By \eqref{E4.17.1}, we see that
\begin{equation}\label{E4.17.2}\tag{E4.17.2}
\delta(V)=-2W\otimes Y+2Y\otimes W+\sum_s\alpha'_sf'_s\otimes g'_s,
\end{equation}
where $\alpha'_s\in k$, $f'_s, g'_s$ are monomials of the form
$X^{i_1}Y^{i_2}Z^{i_3}$.

Now by Theorem \ref{zzthm4.12}, the set $\{X, Y, Z, W\}$ generates
$\gr H$. Also, $\gr H$ is a ${\mathbb Z}^2$-graded coalgebra by
setting $\deg_2X=(1, 0)$, $\deg_2Y=(0, 1)$, $\deg _2Z=(1,1)$ and
$\deg_2 W= (1, 2)$. Notice that $\deg_2 V=(1, 3)$. Therefore, $V$
must be a linear combination of linearly independent elements
$XY^3$, $Y^2Z$, $YW$ of degree $(1,3)$. Hence there are $\beta_i\in
k$ such that
$$V=\beta_1 XY^3+\beta_2 ZY^2+\beta_3 YW$$
or
\begin{align}\label{E4.17.3}\tag{E4.17.3}
\delta
(V)&=\delta(\beta_1XY^3+\beta_2ZY^2+\beta_3YW)\\
\notag &=\beta_3(Y\otimes W+W\otimes Y)+\sum_s\alpha''_sf''_s\otimes
g''_s,
\end{align}
where $\alpha''_s\in k$, $f''_s, g''_s$ are monomials of the form
$X^{i_1}Y^{i_2}Z^{i_3}$. If we compare the coefficient of $W\otimes
Y$ on \eqref{E4.17.2} and \eqref{E4.17.3}, we get $-2=\beta_3$. On
the other hand, if we compare the coefficient of $Y\otimes W$, we
have $2=\beta_3$, which is a contradiction. This completes the
proof.
\end{proof}

\subsection{Primitively-thin Hopf algebra of GK-dimension 4, Part III}
\label{zzsec4.5} In this subsection we deal with the case when
$C=A(\lambda_1,\lambda_2,\alpha)$ and finish the analysis.
Throughout this subsection we assume that
$C=A(\lambda_1,\lambda_2,\alpha)$ where the relations of
$A(\lambda_1,\lambda_2,\alpha)$ are given in Example \ref{zzex4.1};
and that $(\lambda_1,\lambda_2,\alpha)$ is either $(0,0,0)$, or
$(0,0,1)$ or $(1,1,1)$ or $(1,\lambda,0)$ as listed in Proposition
\ref{zzpro4.3}(a).

\begin{lemma}\label{zzlem4.18} Let $u$ and $t$ be given as in
\eqref{E4.9.1}-\eqref{E4.9.2}.
\begin{align}
\label{E4.18.1}\tag{E4.18.1} [u,x\otimes 1&+1\otimes x]
=\alpha(y\otimes x-x\otimes y),\\
\label{E4.18.2}\tag{E4.18.2} [t,x\otimes 1&+1\otimes x]
=\lambda_1(y\otimes x-x\otimes y),\\
\label{E4.18.3}\tag{E4.18.3} [u,y\otimes 1&+1\otimes y]
=\lambda_2(y\otimes x-x\otimes y),\\
\label{E4.18.4}\tag{E4.18.4} [t,y\otimes 1&+1\otimes y] =0.
\end{align}
\end{lemma}

\begin{proof} We use the relations of $A(\lambda_1,\lambda_2,\alpha)$
and note that $[x,y]=0$. By an easy computation, we have
$$\begin{aligned}
\;[u,x\otimes 1]&=[z,x]\otimes
x=(\lambda_1x+\alpha y)\otimes x,\\
[u,1\otimes x]&=-x\otimes [z,x]=-x\otimes (\lambda_1x+\alpha y).
\end{aligned}
$$
Now \eqref{E4.18.1} follows by adding the above.

By a computation,
$$\begin{aligned}
\;[t,x\otimes 1]&=-[z,x]\otimes
y=-(\lambda_1x+\alpha y)\otimes y,\\
[t,1\otimes x]&=y\otimes [z,x]=y\otimes (\lambda_1x+\alpha y),
\end{aligned}
$$
and \eqref{E4.18.2} is obtained by adding the above. The proof of the
last two are similar.

\end{proof}

\begin{lemma}
\label{zzlem4.19} Let $w$ be as in Theorem \ref{zzthm4.12}.
\begin{align} \label{E4.19.1}\tag{E4.19.1}
[w,x]&= -(\theta_1 \alpha +\theta_2 \lambda_1)z+a_{11}x+a_{12}y\\
\label{E4.19.2}\tag{E4.19.2} [w,y]&=-\theta_1\lambda_2
z+a_{21}x+a_{22}y
\end{align}
for some $a_{11},a_{12},a_{21},a_{22}\in k$.
\end{lemma}

\begin{proof} We only prove the first equation and the proof of the
second equation is similar.
$$\begin{aligned}
\delta([w,x])&=\Delta([w,x])-[w,x]\otimes 1-1\otimes [w,x]\\
&=[\Delta(w), x\otimes 1+1\otimes x]-[w,x]\otimes 1-1\otimes [w,x]\\
&=[w\otimes 1+1\otimes w+\theta_1 u+\theta_2 t,
x\otimes 1+1\otimes x]-[w,x]\otimes 1-1\otimes [w,x]\\
&=[\theta_1 u+\theta_2 t, x\otimes 1+1\otimes x]\\
&=\theta_1 [u,x\otimes 1+1\otimes x]+\theta_2 [t,x\otimes 1+1\otimes x]\\
&=(\theta_1 \alpha +\theta_2 \lambda_1)(y\otimes x-x\otimes y)
\qquad\qquad \qquad {\text{by \eqref{E4.18.1}-\eqref{E4.18.2}}}\\
&=-(\theta_1 \alpha +\theta_2 \lambda_1)\delta(z).
\end{aligned}
$$
Therefore $[w,x]+(\theta_1 \alpha +\theta_2 \lambda_1)z$ is a
primitive elements, whence it is of the form $a_{11}x+a_{12}y$ for
some $a_{11},a_{12}\in k$. The assertion follows.
\end{proof}

\begin{lemma}
\label{zzlem4.20} Retain the notation as above. Then $\lambda_2=0$.
Consequently, $(\lambda_1,\lambda_2,\alpha)$ is either $(0,0,0)$, or
$(0,0,1)$ or  $(1,0,0)$.
\end{lemma}

\begin{proof} Since $[x,y]=0$, using Lemma \ref{zzlem4.19} we have
$$\begin{aligned}
0&=[w,[x,y]]=[[w,x],y]+[x,[w,y]]\\
&= [-(\theta_1 \alpha +\theta_2 \lambda_1)z+a_{11}x+a_{12}y,y]+ [x,
-\theta_1\lambda_2 z+a_{21}x+a_{22}y]\\
&=-(\theta_1 \alpha +\theta_2
\lambda_1)[z,y]+\theta_1\lambda_2[z,x]\\
&=-(\theta_1 \alpha +\theta_2 \lambda_1)(\lambda_2
y)+\theta_1\lambda_2(\lambda_1 x+\alpha y)\\
&=-\theta_2\lambda_1\lambda_2 y+\theta_1\lambda_1\lambda_2 x.
\end{aligned}
$$
Since one of $\theta_i$'s is nonzero, $\lambda_1\lambda_2=0$. We
only consider those $(\lambda_1,\lambda_2,\alpha)$'s given in
Proposition \ref{zzpro4.3}(a), therefore, $\lambda_2=0$.
\end{proof}

\begin{lemma}
\label{zzlem4.21} Suppose $\lambda_2=0$. Let $u$ and $t$ be given as
in \eqref{E4.9.1}-\eqref{E4.9.2}.
\begin{align}
\label{E4.21.1}\tag{E4.21.1} [u,z\otimes 1&+1\otimes z]
=-\lambda_1 u+\alpha t-\alpha \delta(xy^2)-\lambda_1(xy\otimes x+x\otimes xy)),\\
\label{E4.21.2}\tag{E4.21.2} [t,z\otimes 1&+1\otimes z]
=-\lambda_1(xy\otimes y+y\otimes xy)-\alpha(y^2\otimes y+y\otimes
y^2),\\
\label{E4.21.3}\tag{E4.21.3} [u, x\otimes y&-y\otimes
x]=\lambda_1(x\otimes xy+xy\otimes x)+\alpha(y\otimes xy+xy\otimes
y),\\
\label{E4.21.4}\tag{E4.21.4} [t, x\otimes y&-y\otimes
x]=-\lambda_1(x\otimes y^2+y^2\otimes x)-\alpha(y\otimes
y^2+y^2\otimes y).
\end{align}
\end{lemma}

\begin{proof} By direct computation, we have
$$\begin{aligned}
\; [u,z\otimes 1]&=(\lambda_1 x+\alpha y) \otimes z-(\lambda_1
x+\alpha y)y\otimes x-(\lambda_1 x
+\alpha y)\otimes xy,\\
[u,1\otimes z]&=-z\otimes(\lambda_1 x+\alpha y)-xy\otimes (\lambda_1
x+\alpha y)-x\otimes(\lambda_1 x+\alpha y)y.
\end{aligned}
$$
Adding up and using definition and \eqref{E4.16.1}, we obtain
\eqref{E4.21.1}. Others are similar, by using definitions and direct
computations.
\end{proof}

\begin{lemma}
\label{zzlem4.22} Suppose $\lambda_2=0$. Then
\begin{equation}\label{E4.22.1}\tag{E4.22.1}
\theta_1(\theta_2 \lambda_1+\theta_1 \alpha)=0.
\end{equation}
Further,
\begin{enumerate}
\item
If $(\lambda_1,\lambda_2,\alpha)=(0,0,0)$, then
$$[w,z]=(a_{11}+a_{22})z+\xi_1 x+\xi_2 y$$
for some $\xi_i\in k$.
\item
If $(\lambda_1,\lambda_2,\alpha)=(1,0,0)$, then
$\theta_1\theta_2=0$. If, moreover, $\theta_1=0$,  then
$$[w,z]=(a_{11}+a_{22})z+w+(-\theta_2)xy^2+\xi_1 x+\xi_2 y$$
for some $\xi_i\in k$. If, moreover, $\theta_2=0$,  and
$$[w,z]=(a_{11}+a_{22})z-w+\xi_1 x+\xi_2 y$$
for some $\xi_i\in k$.
\item
If $(\lambda_1,\lambda_2,\alpha)=(0,0,1)$, then $\theta_1=0$ and
$$[w,z]=(a_{11}+a_{22})z-\frac{2}{3}\theta_2 y^3+\xi_1 x+\xi_2 y$$
for some $\xi_i\in k$.
\end{enumerate}
\end{lemma}

\begin{proof}
$$\begin{aligned}
\delta([w,z])&=\Delta([w,z])-[w,z]\otimes 1-1\otimes [w,z]\\
&=[\Delta(w),\Delta(z)]-[w,z]\otimes 1-1\otimes [w,z]\\
&=[w\otimes 1+1\otimes w+\theta_1 u+\theta_2 t, z\otimes 1+1\otimes
z+(x\otimes y-y\otimes x)]\\
&\quad -[w,z]\otimes 1-1\otimes [w,z],\\
&=[w\otimes 1+1\otimes w, (x\otimes y-y\otimes x)]\\
&\quad +[\theta_1 u+\theta_2 t, z\otimes 1+1\otimes
z+(x\otimes y-y\otimes x)]\\
&=-(\theta_1\alpha+\theta_2\lambda_1)(z\otimes y-y\otimes
z)+(a_{11}+a_{22})(x\otimes y-y\otimes x)\\
&\qquad\qquad \qquad \qquad\qquad\qquad \qquad \qquad\qquad \qquad
{\text{by \eqref{E4.19.1}-\eqref{E4.19.2}}}\\
&\quad +\theta_1(-\lambda_1 u+\alpha t-\alpha
\delta(xy^2)+\alpha(y\otimes xy+xy\otimes y))\\
&\qquad\qquad \qquad \qquad\qquad\qquad \qquad \qquad\qquad \qquad
{\text{by \eqref{E4.21.1}-\eqref{E4.21.3}}}\\
&\quad +\theta_2(-\lambda_1 \delta(xy^2)+\lambda_1(xy\otimes
y+y\otimes xy)-2\alpha (y^2\otimes y+y\otimes y^2)) \\
&\qquad\qquad \qquad \qquad\qquad\qquad \qquad \qquad\qquad \qquad
{\text{by \eqref{E4.21.2}-\eqref{E4.21.4}}}\\
&=-\theta_1\lambda_1u+(2\theta_1\alpha+\theta_2\lambda_1)
t\\
&\quad +(a_{11}+a_{22})\delta(z)+(-\theta_1\alpha-\theta_2\lambda_1)
\delta(xy^2)-\frac{2}{3}\theta_2\alpha\delta(y^3).
\end{aligned}
$$
Since $\delta([w,z])\in C\otimes C$, $[w,z]$ induces a cohomology
class in ${\text{H}}^2(\Omega C)$. By Lemma \ref{zzlem4.14}(c),
$\delta([w,z])$ is a scalar multiple of $\delta(w)$. This implies
that
$$\theta_2(-\theta_1\lambda_1)-\theta_1(2\theta_1\alpha+
\theta_2\lambda_1)=0$$ or, after simplifying, we obtain
\eqref{E4.22.1}.

(a) By the above computation, we have
$\delta([w,z])=(a_{11}+a_{22})\delta(z)$. The assertion follows.

(b) When $(\lambda_1,\lambda_2,\alpha)=(1,0,0)$, \eqref{E4.22.1}
becomes $\theta_1\theta_2=0$ and
$$\delta([w,z])=-\theta_1u+\theta_2t+(a_{11}+a_{22})\delta(z)-\theta_2
\delta(xy^2).$$ Since $\theta_1\theta_2=0$, $-\theta_1u+\theta_2t$
is either $\delta(w)$ or $-\delta(w)$. Hence we have
$$[w,z]=(a_{11}+a_{22})z+c w+(-\theta_2)xy^2+\xi_1 x+\xi_2 y$$
for some $\xi_i\in k$ and $c=\pm 1$, which gives two cases listed in
part (b).

(c) If $(\lambda_1,\lambda_2,\alpha)=(0,0,1)$, then \eqref{E4.22.1}
implies that $\theta_1=0$ and
$$\delta([w,z])=(a_{11}+a_{22})\delta(z)-\frac{2}{3}\theta_2\delta(y^3)$$
and therefore the assertion follows.
\end{proof}

\begin{theorem}\label{zzthm4.23}
Let $H$ be a connected Hopf algebra of GK-dimension $4$ with $\dim_k
P(H)=2$. Then $H$ must be isomorphic to one of the Hopf algebras
listed in Example \ref{zzex4.4}, \ref{zzex4.5}, \ref{zzex4.6} and
\ref{zzex4.7}.
\end{theorem}

\begin{proof}
Retain the notation in Theorem \ref{zzthm4.12}. By Proposition
\ref{zzpro4.17} and Proposition \ref{zzpro4.3}, $C$ must be
isomorphic to $A(\lambda_1, \lambda_2,\alpha)$, where the possible
choices of $(\lambda_1, \lambda_2,\alpha)$ are listed in Proposition
\ref{zzpro4.3}(a). By Lemma \ref{zzlem4.20}, $\lambda_2=0$. Hence
$C$ is either $A(0,0,0)$, or $A(1,0,0)$, or $A(0,0,1)$.

Case 1: $C=A(0,0,0)$. By Lemma \ref{zzlem4.19} and
\ref{zzlem4.22}(a), we have
$$\begin{aligned}
\; [w,x]&= a_{11}x+a_{12}y\\
[w,y]&=a_{21}x+a_{22}y\\
[w,z]&=(a_{11}+a_{22})z+\xi_1 x+\xi_2 y
\end{aligned}
$$
for some $a_{11},a_{12},a_{21},a_{22},\xi_1,\xi_2\in k$. Together
with the coalgebra given in Theorem \ref{zzthm4.12}, this is the
Hopf algebra described in Example \ref{zzex4.4}.

Case 2: $C=A(1,0,0)$. By Lemma \ref{zzlem4.22}(b), there are two
cases to consider. If $\theta_1=0$, by dividing $w$ with $\theta_2$
we can assume that $\theta_2=1$. In this setting, by Lemmas
\ref{zzlem4.19} and \ref{zzlem4.22}(b), we have
$$\begin{aligned}
\; [w,x]&=-z+a_{11}x+a_{12}y\\
[w,y]&=a_{21}x+a_{22}y\\
[w,z]&=(a_{11}+a_{22})z+w-xy^2+\xi_1 x+\xi_2 y
\end{aligned}
$$
for some $a_{11},a_{12},a_{21},a_{22},\xi_1,\xi_2\in k$. Applying
$[w,-]$ to $[z,y]=0$, one sees that
$$\begin{aligned}
0&=[w,[z,y]]=[[w,z],y]+[z,[w,y]]\\
&=[(a_{11}+a_{22})z+w-xy^2+\xi_1 x+\xi_2 y,y]+[z,a_{21}x+a_{22}y]\\
&=[w,y]+a_{21}[z,x]=a_{21}x+a_{22}y+a_{21}x.
\end{aligned}
$$
Thus $a_{21}=a_{22}=0$. Applying $[w,-]$ to $[z,x]=x$ and using
$a_{21}=a_{22}=0$, one sees that
$$\begin{aligned}
\; [w,x]&=[w,[z,x]]\\
&=[[w,z],x]+[z,[w,x]]\\
&=[a_{11}z+w-xy^2+\xi_1 x+\xi_2 y,x]+[z,-z+a_{11}x+a_{12}y]\\
&=a_{11}[z,x]+[w,x]+a_{11}[z,x]=[w,x]+2a_{11}x
\end{aligned}
$$
which implies that $a_{11}=0$. By setting $z_{new}=z-a_{12}y$, we
can make $a_{12}=0$. By setting $w_{new}=w+\frac{1}{2} \xi_1 x+\xi_2
y$, we can make $\xi_1=\xi_2=0$. New variables $z$ and $w$ still
satisfy \eqref{E4.12.1} and \eqref{E4.12.2}, respectively (for
$\theta_1=0$). Therefore this is the Hopf algebra described in
Example \ref{zzex4.7}.

If $\theta_2=0$, by dividing $w$ with $\theta_1$ we can assume that
$\theta_1=1$. In this setting, by Lemmas \ref{zzlem4.19} and
\ref{zzlem4.22}(b), we have
$$\begin{aligned}
\; [w,x]&=a_{11}x+a_{12}y\\
[w,y]&=a_{21}x+a_{22}y\\
[w,z]&=(a_{11}+a_{22})z-w+\xi_1 x+\xi_2 y
\end{aligned}
$$
for some $a_{11},a_{12},a_{21},a_{22},\xi_1,\xi_2\in k$. Applying
$[w,-]$ to $[z,y]=0$, one sees that
$$\begin{aligned}
0&=[w,[z,y]]=[[w,z],y]+[z,[w,y]]\\
&=[(a_{11}+a_{22})z-w+\xi_1 x+\xi_2 y,y]+[z,a_{21}x+a_{22}y]\\
&=-[w,y]+a_{21}[z,x]=-a_{22}y.
\end{aligned}
$$
Thus $a_{22}=0$. Applying $[w,-]$ to $[z,x]=x$ and using $a_{22}=0$,
one sees that
$$\begin{aligned}
\; [w,x]&=[w,[z,x]]\\
&=[[w,z],x]+[z,[w,x]]\\
&=[a_{11}z-w+\xi_1 x+\xi_2 y,x]+[z,a_{11}x+a_{12}y]\\
&=a_{11}[z,x]-[w,x]+a_{11}[z,x]=-[w,x]+2a_{11}x
\end{aligned}
$$
which implies that $a_{12}=0$. This is the Hopf algebra described in
Example \ref{zzex4.5}.

Case 3: $C=A(0,0,1)$. By Lemma \ref{zzlem4.22}(c), $\theta_1=0$. By
dividing $w$ with $\theta_2$ we can assume that $\theta_2=1$. In
this setting, by Lemmas \ref{zzlem4.19} and \ref{zzlem4.22}(c), we
have
$$\begin{aligned}
\; [w,x]&=a_{11}x+a_{12}y\\
[w,y]&=a_{21}x+a_{22}y\\
[w,z]&=(a_{11}+a_{22})z-\frac{2}{3} y^3+\xi_1 x+\xi_2 y
\end{aligned}
$$
for some $a_{11},a_{12},a_{21},a_{22},\xi_1,\xi_2\in k$. Applying
$[w,-]$ to $[z,y]=0$, one sees that
$$\begin{aligned}
0&= [w,[z,y]]=[[w,z],y]+[z,[w,y]]\\
&=[(a_{11}+a_{22})z-\frac{2}{3} y^3+\xi_1 x+\xi_2
y,y]+[z,a_{21}x+a_{22}y]\\
&=a_{21}y.
\end{aligned}
$$
Hence $a_{21}=0$. Applying $[w,-]$ to $[z,x]=y$, one sees that
$$\begin{aligned}
a_{22}y&=[w,y]=[w,[z,x]]=[[w,z],x]+[z,[w,x]]\\
&=[(a_{11}+a_{22})z-\frac{2}{3} y^3+\xi_1 x+\xi_2
y,x]+[z,a_{11}x+a_{12}y]\\
&=(a_{11}+a_{22})y+a_{11}y.
\end{aligned}
$$
Hence $a_{11}=0$. This is the Hopf algebra described in Example
\ref{zzex4.6}.
\end{proof}

\subsection{Proof of the main result}
\label{zzsec4.6} With the help of the last few subsections, we are
able to deliver the main theorem of the paper.

\begin{theorem}[Theorem \ref{zzthm0.3}]
\label{zzthm4.24} Let $H$ be a connected Hopf algebra of
GK-dimension four over an algebraically closed field of
characteristic zero. Then $H$ is isomophic to one of following.
\begin{enumerate}
\item
Enveloping algebra $U({\mathfrak g})$ over a Lie algebra ${\mathfrak
g}$ of dimension $4$. Note that all 4-dimensional Lie algebras over
the complex numbers ${\mathbb C}$ are listed in the book
\cite[Theorem 1.1(iv), page 209]{OV}.
\item
Enveloping algebra $U(L)$ over an anti-cocommutative CLA $L$ of
dimension $4$. All anti-cocommutative coassociative Lie algebras of
dimension $4$ are classified in Theorem \ref{zzthm3.5}.
\item
Primitively-thin Hopf algebras of GK-dimension four. All
Primitively-thin Hopf algebras of GK-dimension four are classified
in Theorem \ref{zzthm4.23}.
\end{enumerate}
\end{theorem}

\begin{proof} By Lemma \ref{zzlem1.3}(e), $p(H)\ge 2$. Since
$U(\mathfrak{g})$ embeds in $H$ where $\mathfrak{g}=P(H)$, we
have $p(H)\le 4$.

If $p(H)=4$, by Theorem \ref{zzthm2.7}, $H\cong
U(\mathfrak{g})$ where $\mathfrak{g}=P(H)=P_2(H)$. This is case (a).

If $p(H)=3$, by Theorem \ref{zzthm2.7}, $H$ is isomorphic to the
enveloping algebra $U(L)$ over an anti-cocommutative CLA $L$ of
dimension $4$. Anti-cocommutative CLAs of dimension $4$ are
classified in Theorem \ref{zzthm3.5}.

If $p(H)=2$, by definition, $H$ is a primitively-thin Hopf algebras
of GK-dimension four, which are classified in Theorem
\ref{zzthm4.23}.
\end{proof}

\begin{corollary}
\label{zzthm4.25} Let $H$ be a connected Hopf algebra of dimension
at most 4. Then, as an algebra, $H$ is isomorphic to $U(\mathfrak
g)$ for some Lie algebra ${\mathfrak g}$.
\end{corollary}

\begin{proof} This is clear if $\GKdim H\leq 3$. For $\GKdim H=4$,
the assertion is clear for the case when $p(H)\geq 3$. The only case
left is when $p(H)=2$, in which the assertion was checked
case-by-case in Examples \ref{zzex4.4}-\ref{zzex4.7}.
\end{proof}

\providecommand{\bysame}{\leavevmode\hbox to3em{\hrulefill}\thinspace}
\providecommand{\MR}{\relax\ifhmode\unskip\space\fi MR }
\providecommand{\MRhref}[2]{%

\href{http://www.ams.org/mathscinet-getitem?mr=#1}{#2} }
\providecommand{\href}[2]{#2}

\end{document}